\documentclass[12pt]{amsart}

\usepackage[colorlinks=false]{hyperref}
\usepackage[framed]{matlab-prettifier}
\usepackage{amssymb,amsthm,amsfonts,amsmath,mathtools}
\usepackage{lineno,hyperref}
\usepackage[a4paper,bindingoffset=0.2in,%
            left=1in,right=1in,top=1in,bottom=1in,%
            footskip=.25in]{geometry}
\usepackage{color}
\usepackage{natbib}
\setcitestyle{numbers,sort,open={[},close={]}}
\hypersetup{
    colorlinks=false,
		linkbordercolor = {white},
		citebordercolor = {white},
    filecolor=magenta,      
    urlcolor=red,
    }

\newtheorem{thm}{Theorem}[section]

\newtheorem{prop}[thm]{Proposition}

\theoremstyle{definition}
\newtheorem{defn}{Definition}[section]

\theoremstyle{remark}
\newtheorem{remark}{Remark}[section]
\usepackage{amsopn}
\DeclareMathOperator{\erf}{erf}

\ifpdf
\hypersetup{
  pdftitle={Variable-order fractional calculus: a change of perspective},
  pdfauthor={R. Garrappa, A. Giusti, and F. Mainardi}
}
\fi

\title[Variable-order fractional calculus] {Variable-order fractional calculus: \\a change of perspective}

\thanks{Accepted version of the paper:  {\color{blue}\href{https://doi.org/10.1016/j.cnsns.2021.105904}{R. Garrappa, A. Giusti, and F. Mainardi}, \emph{Variable-order fractional calculus: a change of perspective},  \href{https://doi.org/10.1016/j.cnsns.2021.105904}{Commun. Nonlinear Sci. Numer. Simul., 2021, 102, 105904}, \href{https://doi.org/10.1016/j.cnsns.2021.105904}{doi:10.1016/j.cnsns.2021.105904}} (released under a CC-BY-NC-ND license).}

\author{Roberto Garrappa}
\address{Roberto Garrappa: Department of Mathematics, University of Bari, Via E. Orabona 4, 70125 Bari, Italy.  Member of the INdAM Research group GNCS}
\email{roberto.garrappa@uniba.it}
\urladdr{www.dm.uniba.it/members/garrappa/main} % Delete if not wanted.

\author{Andrea Giusti}
\address{Andrea Giusti: Physics and Astronomy Department, Bishop’s University, Sherbrooke, Canada - Institute of Theoretical Physics, ETH Zurich, Zurich, Switzerland.  Member of the INdAM Research group GNFM}
\email{agiusti@phys.ethz.ch}

\author{Francesco Mainardi}
\address{Francesco Mainardi: Physics and Astronomy Department, University of Bologna, Bologna, Italy.  Member of the INdAM Research group GNFM}
\email{Francesco.Mainardi@bo.infn.it}

\hypersetup{pdfauthor={R.Garrappa, A.Giusti, F.Mainardi},pdftitle={Variable-order fractional calculus: a change of perspective}}

\begin{document}

\begin{abstract}
Several approaches to the formulation of a fractional theory of calculus of ``variable order'' have appeared in the literature over the years. Unfortunately, most of these proposals lack a rigorous mathematical framework. We consider an alternative view on the problem, originally proposed by G. Scarpi in the early seventies, based on a naive modification of the representation in the Laplace domain of standard kernels functions involved in (constant-order) fractional calculus. We frame Scarpi's ideas within recent theory of General Fractional Derivatives and Integrals, that mostly rely on the Sonine condition, and investigate the main properties of the emerging variable-order operators. Then, taking advantage of powerful and easy-to-use numerical methods for the inversion of Laplace transforms of functions defined in the Laplace domain, we discuss some practical applications of the variable-order Scarpi integral and derivative.
\end{abstract}

\maketitle

\def\Rset{{\mathbb{R}}}
\def\Cset{{\mathbb{C}}}
\def\Zset{{\mathbb{Z}}}
\def\Nset{{\mathbb{N}}}
\def\eu{{\ensuremath{\mathrm{e}}}}
\def\iu{{\ensuremath{\mathrm{i}}}}
\def\du{{\ensuremath{\mathrm{d}}}}
\def\texttiny#1{{\text{\tiny{#1}}}}
\def\DC{{}^{\texttiny{C}}\! D}
\def\DR{{}^{\texttiny{RL}}\! D}
\def\DG{{}^{\texttiny{GL}}\! D}
\def\DL{{}^{\texttiny{L}}\! D}
\def\DS{{}^{\texttiny{S}}\! D}
\def\IS{{}^{\texttiny{S}}\! I}
\def\IRL{{}^{\texttiny{RL}}\! I}
\newcommand{\be}{\begin{eqnarray}}
\newcommand{\ee}{\end{eqnarray}}

\section{Introduction}%\label{S:Introduction}

Derivatives and integrals of fractional (i.e., non-integer) order are among the most fashionable tools for modeling phenomena featuring  persistent memory effects (i.e., non-localities in time). Since many physical systems are characterized by dynamics involving memory effects whose behaviour changes over time, even transitioning from a fractional order to another, the interest for fractional operators soon moved to their variable-order counterparts. Needless to say that the compelling practical implications of these variable-order objects come at the price of a more involved mathematical characterization.

A naturally looking variable-order generalization of standard fractional derivatives is obtained by replacing the constant order $\alpha$ with a function $\alpha \, : [0, T] \subset \, \Rset ^{+} \, \to \, (0,1)$ in the Riemann-Liouville (RL) integral, {\em i.e.},
\begin{equation}
I^{\alpha(t)}_0 f(t) = \frac{1}{\Gamma\bigl(\alpha(t)\bigr)} \int_0^t (t-\tau)^{\alpha(t)-1} \, f(\tau) \du \tau,
\end{equation}
possibly coupled to the RL-like variable-order derivative
\begin{equation}
D^{\alpha(t)}_0 f(t) = \frac{1}{\Gamma\bigl(1-\alpha(t)\bigr)} \frac{\du}{\du t} \int_0^t (t-\tau)^{-\alpha(t)} f(\tau) \du \tau ,
\end{equation}
where we restricted $0<\alpha(t)<1$ for the sake of simplicity. However, the mathematical characterization of fractional calculus based on these operators is rather problematic since $D^{\alpha(t)}_0$ {\em does not} necessarily act as the left-inverse for $I^{\alpha(t)}_0$ \cite{Samko1995a,SamkoRoss1993}. This property is, nonetheless, recovered in the constant-order limit of the theory.

Over the years, several proposals for fractional variable-order operators have appeared in the literature; for instance, we recall the works by Bohannan \cite{Bohannan2017}, Coimbra \cite{Coimbra2003}, Ingman and Suzdalnitsky  \cite{IngmanSuzdalnitsky2004}, Kobelev et al. \cite{KobelevKobelevKlimontovich2003b,KobelevKobelevKlimontovich2003a}, Lorenzo and Hartley  \cite{LorenzoHartley2002},    Pedro et al. \cite{PedroKobayashiPereiraCoimbra2008}, Sierociuk et al. \cite{SierociukMaleszaMacias2015}, Sun et al. \cite{SunChenChen}. For a comprehensive review of the literature we refer the interested reader to Ortigueira et al. \cite{OrtigueiraValerioMachado2019}, Samko \cite{Samko2013} and Sun et al. \cite{SunChangZhangChen2019}. However, despite showing seemingly useful for physical applications, these definitions face the conceptual mathematical problems discussed by Samko and Ross  \cite{Samko1995a,SamkoRoss1993}. 

It is worth noting that, despite the complications discussed above, variable-order  methods were employed by Zheng et al.  \cite{ZhengWangFu2020,ZhengWangFu2021} in an attempt of overcoming the inconsistencies of some families of regular-kernel operators (in this regard, see  aslo \cite{AngstmannJacobsByronHenryXu2020,DiethlemGarrappaGiustiStynes2020,Hanyga2020} for a detailed discussion). 

From a numerical perspective, methods for solving variable-order fractional differential equations (FDEs) have also been analysed to some extent;
as a general reference we mention here the works by Chen et al. \cite{ChenLiuBurrage2014}, Tavares et al. \cite{TavaresAlmeidaTorres2016}, Zhuang et al. \cite{ZhuangLiuAnhTurner2009} and some of the several papers by Karniadakis et al. \cite{ZayernouriKarniadakis2015,ZengZhangKarniadakis2015,ZhaoSunKarniadakis2015}.

Viscoelasticity is the perfect playground for variable-order operators, as certain known scenarios display peculiar transitions from an order to another as a function time (see, {\em e.g.}, \cite{Esmonde2020}, \cite{RamirezCoimbra2007} or \cite{MG-2,MG-3,MG-1}). Further, in recent years variable-order fractional calculus has found some applications also in control theory \cite{Bahaa2017, OstalczykDuchBrzezinskiSankowski2015} as well as in modelling aggregation of particles in living cells \cite{FedotovHanZubarevJohnstonAllan2021}. It is also worthwhile to be mentioned the pioneering work by  Checkin, Gorenflo and Sokolov  \cite{ChechkinGorenfloSokolov2005} in which a time-fractional diffusion equation with time-fractional derivative whose order varies  in space is derived starting from the continuous time random walk scheme; a problem for which the  asymptotic representation of the solution has been recently investigated in \cite{FedotovHan2019}. For a review of some of the latest applications of variable-order fractional operators in natural sciences we refer the interested reader to \cite{PatnaikHollkampSemperlotti2020}.

To the best of our knowledge,
the Italian engineer Giambattista Scarpi was however the first to propose  \cite{Scarpi1972a,Scarpi1972b,Scarpi1973}, in the early seventies, the use of time-fractional derivatives with a time-dependent order. 
Scarpi's work was inspired by an early 
model by Smit and de Vries \cite{SmitVries1970}
which was aimed at providing a theoretical framework for materials showing features intermediate between solids and liquids. Notably, the approach proposed by G. Scarpi was not based on a naive replacement, in the kernel of some fractional  derivative, of the constant order  $\alpha$  with a  variable-order function $\alpha(t)$. The procedure proposed by Scarpi, instead, acts in a more subtle way at the level of the Laplace transform (LT) domain (on a different basis, however, with respect to the operators proposed by Coimbra in \cite{ Coimbra2003}) and constitutes an interesting novelty with respect to more traditional approaches. 

Despite the boom that the active research on fractional calculus has been experiencing for the last decade, so far Scarpi's approach has been mostly overlooked (except for a very recent contribution by Cuesta and Kirane  \cite{CuestaKirane2020} of which we are aware thanks to a private communication).

If, on the one hand, Scarpi's works were the first to introduce this peculiar approach to variable-order theories, on the other hand, they are solely focused on physical properties and implications of the proposed methods. In other words, the mathematical foundations supporting these object were not analyzed in details. Additionally, the operators proposed by Scarpi require reliable numerical techniques for handling the inversion of the LT, which were not available at the time of publication of Scarpi's seminal works.

Recently, much effort has been devoted, particularly by Yuri Luchko \cite{Luchko2021_FCAA,Luchko2021_Mathematics,Luchko2021_Symmetry}, to a mathematically sound formulation of a theory of general fractional integrals and derivatives. Such a theory
is aimed at characterizing classes of operators that satisfy some generalizations of the fundamental theorem of calculus by using the Sonine equation \cite{Sonine1884} as guiding principle. This novel approach has the merit of relaxing some of the conditions of Kochubei's general fractional calculus \cite{Kochubei2011,Kochubei2019a,Kochubei2019b} (see also \cite{LuchkoYamamoto2020}), thus encompassing a larger class of non-local operators. The key feature of this classification consist in the fact that it relies upon the Laplace-domain representation of these general fractional operators, thus providing the perfect tool set for designing a robust mathematical framework for Scarpi's ideas.

On the numerical side, the several advancements in the field of the numerical inversion of the LT, among which we recall the contribution by Weidemann and Trefethen \cite{WeidemanTrefethen2007}, provide us with the machinery needed to implement Scarpi's ideas to their fullest.      

It has now come the time to bring Scarpi's variable-order fractional calculus into the spotlight, precisely characterizing its mathematical foundations and highlighting its potential as modelling tool by taking advantage of modern numerical methods. 

This work is organized as follows. In Section \ref{S:Scarpi}, after recalling some basics of fractional calculus, we introduce the notions of 
the Scarpi derivative and integral.
In Section \ref{S:NecessaryAssumptions} we frame Scarpi's theory within a more general theoretical scheme for fractional calculus, based on the Sonine equation, and we investigate possible assumptions on the variable-order functions $\alpha(t)$. In Section \ref{S:TransitionFunctions} we consider some instructive examples operators obtained for some variable-order functions $\alpha(t)$ and Section \ref{S:S_Relaxation} is devoted the solution of the relaxation equation with the Scarpi derivative. Some considerations about higher-order operators are provided in Section \ref{S:HigherOrderOperators} and, finally, in Section \ref{S:FinalConsiderations} we provide some concluding remarks. Note that the method used to invert numerically the LT, allowing the investigation of Scarpi's fractional operators, is discussed in Appendix \ref{S:NumericalInversionLT}.

\section{Scarpi's variable-order fractional calculus}\label{S:Scarpi}

In order to introduce, and further develop, Scarpi's ideas on variable-order derivatives we preliminary recall some background materials on fractional integrals and derivatives.

\subsection{Preliminaries}

In this work we consider functions which are absolutely continuous on some interval $[0,T]$, {\em i.e.} $f \in AC[0,T]$. This is a not particularly restrictive assumption and it means that $f$ is  differentiable almost everywhere in $[0,T]$, with $f'\in L^1[0,T]$, where $L^1[0,T]$ is the usual space of Lebesgue-integrable functions on $[0,T]$, and 
\begin{equation}\label{eq:AbsoluteContinous}f(t) = f(0) + \int_0^t f'(s)\du s , \quad t \in [0,T] .
\end{equation}

The standard Dzhrbashyan-Caputo notion of fractional derivative of order $0<\alpha<1$, commonly referred to simply as Caputo derivative, is defined in terms of the weakly-singular Volterra-type integro-differential operator
\begin{equation}\label{eq:DCaputo}
\DC^{\alpha}_0 f(t) =  \int_0^t \phi(t-\tau) f'(\tau) \du \tau, \quad
\phi(t) = \frac{t^{-\alpha}}{\Gamma(1-\alpha) } .
\end{equation}
The defining property of $\DC^{\alpha}_0$ is that it acts as the {\em left-inverse} of the RL integral
\begin{equation}\label{eq:RLIntegral}
    \IRL^{\alpha}_0 f(t) = \int_0^t  \psi(t-\tau) f(\tau) \du \tau , \quad \psi(t) = \frac{1}{\Gamma(\alpha)} t^{\alpha-1} ,
\end{equation}
see {\em e.g.} \cite{Diethelm2010,KilbasSrivastavaTrujillo2006}.\footnote{The function $\psi(t)$ is known as the Gel'fand-Shilov kernel \cite{GelfandShilov1964,GorenfloMainardi1997,Mainardi2010}.} 
In other words, one has that  
$$
\DC^{\alpha}_0 \, \IRL^{\alpha}_0 f(t) = f(t) \quad \mbox{and} \quad \IRL^{\alpha}_0 \, \DC^{\alpha}_0 f(t) = f(t) - f(0) \, ,$$
thus implementing a sort of fundamental theorem of fractional calculus \cite{Luchko2020_FCAA}. Basically, the Caputo derivative $\DC^{\alpha}_0$ was introduced to provide a regularization of the RL one
\begin{equation}\label{eq:DRL}
\DR^{\alpha}_0 f(t) =  \frac{\du}{\du t} \int_0^t \phi(t-\tau) f(\tau) \du \tau, \quad
\phi(t) = \frac{t^{-\alpha}}{\Gamma(1-\alpha)} \, . 
\end{equation}

Indeed, the Caputo derivative allows to write fractional differential equations (FDEs) of order $0<\alpha<1$ coupled to the usual initial value conditions at the origin, {\em i.e.}, involving just integer order derivatives. Since these initial value problems have a more straightforward physical interpretation, in this work we focus on regularized Caputo-like derivatives, though this is done without loss of generality since recasting the arguments presented here in the RL framework does not involve any particular complication.

Before moving on with the analysis of the Scarpi derivative it is important to recall some important properties of standard fractional operators. Specifically, we recall that the LT of the kernels involved in the definitions of the Caputo derivative (\ref{eq:DCaputo}) and of the RL integral (\ref{eq:RLIntegral}) are 
\begin{equation}\label{eq:LTKernelStandard}
    \Phi(s) \coloneqq {\mathcal L} \Bigl( \phi(t) \, ; \, s \Bigr) = s^{\alpha-1}
    , \quad
    \Psi(s) \coloneqq {\mathcal L} \Bigl( \psi(t) \, ; \, s \Bigr) = \frac{1}{s^{\alpha}} 
\end{equation}
and, by taking advantage of these LTs, one finds that, assuming that the function $f(t)$ admits the LT $F(s)$, the LTs of (\ref{eq:DCaputo}) and (\ref{eq:RLIntegral}) are
\begin{equation}
{\mathcal L} \Bigl(
\DC^{\alpha}_0 f(t) \, ; \, s \Bigr) = s^{\alpha} F(s) - s^{\alpha-1} f(0)
, \quad
{\mathcal L} \Bigl(
\IRL^{\alpha}_0 f(t) \, ; \, s \Bigr) = \frac{1}{s^{\alpha}} F(s) \, .
\end{equation}

\subsection{A variable-order fractional derivative}
In order to provide a variable-order  generalization of (\ref{eq:DCaputo}) 
we consider a function
\[
\alpha(t):[0,T]\to (0,1)
\]
assumed to be locally integrable on $[0,T]$. The restriction on the image of $\alpha(t)$ to $(0,1)$ is done to avoid further technical complications.

The main idea by Scarpi presented in the pioneering works \cite{Scarpi1972a,Scarpi1972b,Scarpi1973} 
was to define a fractional derivative of variable order $\alpha(t)$ by generalizing the representation (\ref{eq:LTKernelStandard}) in the LT domain of the kernel $\phi(t)$, rather than in the time domain. 

If one considers the constant function $\alpha(t) \equiv \alpha$, $t > 0$, its LT is $A(s)=\alpha/s$ and hence one can trivially infer that $\Phi(s)$ and $\Psi(s)$ in (\ref{eq:LTKernelStandard}), can be recast in terms of $A(s)$ as 
\[
\Phi(s)=s^{sA(s)-1} \, \quad \Psi(s)=s^{-sA(s)} .
\]

Thus, Scarpi's idea consists in extending this simple argument to any non-constant locally integrable function $\alpha(t)$ with LT  
\[
    A(s) \coloneqq {\mathcal L} \Bigl( \alpha(t) \, ; \, s \Bigr) = \int_0^{\infty} \eu^{-st} \alpha(t) \du t ,
\]
and define a variable-order derivative by means of a  convolution similar to those in (\ref{eq:DCaputo}) and (\ref{eq:DRL}). We now formalize this idea in the framework of the theory of Generalized Fractional Derivatives \cite{Luchko2021_Mathematics,Luchko2021_Symmetry,Kochubei2011,Luchko2020_FCAA}.

\begin{defn}\label{defn:ScarpiDerivative}
Let $\alpha \, :  \, [0,T]\to(0,1)$ be a locally integrable function, with $A(s)$ being its  LT, and let $f\in L_1[0,T]$. The regularized (Caputo–Dzhrbashyan type) Scarpi fractional derivative $\DS^{\alpha(t)}_0$  of variable order $\alpha(t)$ is defined as 
\begin{equation}\label{eq:ScarpiDerivativeConvlutionFirstDefinition}
\DS^{\alpha(t)}_0 f(t) \coloneqq \frac{\du}{\du t} \int_0^t \phi_{\alpha}(t-\tau) f(\tau) \du \tau - \phi_{\alpha}(t) f(0)
, \quad t \in (0,T] ,
\end{equation}
where the kernel function $\phi_{\alpha}(t)$ is the inverse LT
\begin{equation}\label{eq:KernelDerivative}
\phi_{\alpha}(t) \coloneqq {\mathcal L}^{-1} \Bigl( \Phi_{\alpha}(s) \, ; \, t \Bigr) , \quad \Phi_{\alpha}(s) = s^{s A(s)-1}. 
\end{equation}
\end{defn}

From the practical perspective it is often useful to recast a fractional operator in the standard Caputo representation for fractional derivatives. Thus,

\begin{prop}\label{prop:ScarpiDerivative}
Let $\alpha \, :  \, [0,T]\to(0,1)$ be a locally integrable function, with $A(s)$ denoting its LT, and let $\phi_{\alpha}(t)$ be the inverse LT of $\Phi_{\alpha}(s) = s^{s A(s)-1}$. If $f\in AC[0,T]$ then 
\begin{equation}\label{eq:ScarpiDerivativeConvlution}
\DS^{\alpha(t)}_0 f(t) = \int_0^t \phi_{\alpha}(t-\tau) f'(\tau) \du \tau
, \quad t \in [0,T] ,
\end{equation}
almost everywhere.
\end{prop}

\begin{proof}
Since $f\in AC[0,T]$, in view of (\ref{eq:AbsoluteContinous}) the integral in (\ref{eq:ScarpiDerivativeConvlutionFirstDefinition}) reads 
\[
    \int_0^t \phi_{\alpha}(t-\tau) f(\tau) \du \tau
    = \int_0^t \phi_{\alpha}(t-\tau) f(0) \du \tau +
     \int_0^t \phi_{\alpha}(t-\tau) \int_0^{\tau} f'(s) \du s \du \tau \, .
\]
Exchanging the order of integration in the second piece one finds  
\[
    \int_0^t \phi_{\alpha}(t-\tau) f(\tau) \du \tau
    = \int_0^t \phi_{\alpha}(t-\tau)  \du \tau f(0) +
     \int_0^t \left(\int_0^{\tau}   \phi_{\alpha}(\tau-s) f'(s) \du s \right)\du \tau ,
\]
then differentiating both sides with respect to $t$ one gets 
\[
    \frac{\du}{\du t} \int_0^t \phi_{\alpha}(t-\tau) f(\tau) \du \tau
    = \phi_{\alpha}(t) f(0) +
\int_0^{t} \phi_{\alpha}(t-s) f'(s) \du s ,
\]
that concludes the proof.
\end{proof}

Clearly, the Scarpi derivative reduces to the standard Caputo one when $\alpha(t)$ becomes constant. Furthermore, from well-known properties of the LT one immediately finds that
\begin{equation}\label{eq:ScarpiDerDefinitionLT}
   {\mathcal L} \Bigl( \DS^{\alpha(t)}_0 f(t) \, ; \, s \Bigr) =  s^{s A(s)} F(s) - s^{s A(s)-1} f(0) \,  .
\end{equation}
However, finding an explicit representation of the kernel $\phi_{\alpha}(t)$ is not always possible and in one of the following sections we will explore some computational approaches to this problem. 

\begin{remark}
Note that Scarpi did not consider a variable-order derivative regularized in the Caputo–Dzhrbashyan way in his 1972 and 1973 works. Nonetheless, since such a regularization has relevant implications we believe that it is of grater interest to deal with this formulation of the Scarpi derivative.
\end{remark}

\subsection{A corresponding variable-order fractional integral}

It is of interest, especially for solving differential equations, to find an integral operator $\IS^{\alpha(t)}_0$ of convolution type, with some kernel $\psi_{\alpha}(t)$, such that the fundamental theorem of fractional calculus holds also for the Scarpi derivative, namely 
\begin{equation}\label{eq:FundamentalTheorem}
\DS^{\alpha(t)}_0 \IS^{\alpha(t)}_0 f(t) = f(t)
, \quad
\IS^{\alpha(t)}_0 \DS^{\alpha(t)}_0 f(t) = f(t) - f(0).
\end{equation} 
For this to be true the two kernels $\phi_{\alpha}(t)$ and $\psi_{\alpha}(t)$ must satisfy the Sonine equation  \cite{Luchko2021_Symmetry,Sonine1884,Kochubei2011,SamkoCardoso2003b,SamkoCardoso2003a}
\begin{equation}\label{eq:SonineEquation}
    \int_0^t \phi_{\alpha}(t-\tau) \psi_{\alpha}(\tau) = 1, \quad t > 0 , 
\end{equation}
and, in this case,  $\phi_{\alpha}(t)$ and $\psi_{\alpha}(t)$ are said to form a Sonine pair. Sonine pairs have been 
extensively studied in the literature, see {\em e.g.}, \cite{SamkoCardoso2003b,SamkoCardoso2003a}. 

Given a generic function $\phi_{\alpha}(t)$, finding the corresponding  $\psi_{\alpha}(t)$ such that the two functions form a Sonine pair is not trivial. However, this problem simplifies substantially
when working in the Laplace domain.

\begin{prop}
Let $\alpha \, :  \, [0,T]\to(0,1)$ be a locally integrable function, let $A(s)$ denote the LT of  $\alpha(t)$, and let $f\in L^1 [0,T]$. 
The integral operator 
\begin{equation}\label{eq:ScarpiIntegralConvolution}
\IS^{\alpha(t)}_0 f(t) = \int_0^{t} \psi_{\alpha}(t-\tau) f(\tau) \du \tau ,
\end{equation}
satisfies the conditions in (\ref{eq:FundamentalTheorem}) when
\begin{equation}\label{eq:KernelIntegral}
\psi_{\alpha}(t) \coloneqq {\mathcal L}^{-1} \Bigl( \Psi_{\alpha}(s) \, ; \, t \Bigr) , \quad \Psi_{\alpha}(s) = s^{-s A(s)}. 
\end{equation}
\end{prop}

\begin{proof}
If $\Psi_{\alpha}(s) = s^{-s A(s)}$, then $\phi_{\alpha}(t)$ and $\psi_{\alpha}(t)$ form a Sonine pair. Indeed, the Sonine equation (\ref{eq:SonineEquation}) in the Laplace domain reads  
\begin{equation}\label{eq:SonineEqationLT}
    \Phi_{\alpha}(s) \Psi_{\alpha}(s) = \frac{1}{s}
\end{equation}
which is trivially satisfied because of the definition of the Scarpi derivative that requires 
$\Phi_{\alpha}(s) = s^{s A(s)-1}$.

\end{proof}

It is worth mentioning that given two functions $\alpha ,\beta :[0,T]\to(0,1)$ a commutative index law 
\[
\IS^{\alpha(t)}_0 \IS^{\beta(t)}_0 f(t) = \IS^{\beta(t)}_0 \IS^{\alpha(t)}_0 f(t) = \IS^{\alpha(t)+\beta(t)}_0 f(t)
\]
can be inferred from Eqs (44) and (45) in \cite{Luchko2021_Mathematics}.

\section{Some necessary assumptions}\label{S:NecessaryAssumptions}

Clearly, not all transition functions $\alpha(t)$ will allow for a suitable definition of a pair of Scarpi-type variable-order fractional operators. In other words, not all $\alpha(t)$ are such that the corresponding kernels $\{\phi_{\alpha}(t) , \psi_{\alpha}(t)\}$ form a Sonine pair and hence $\{\DS^{\alpha(t)}_0 , \IS^{\alpha(t)}_0\}$ satisfy the fundamental theorem of calculus  (\ref{eq:FundamentalTheorem}).

Following the arguments by Samko and Cardoso in \cite{SamkoCardoso2003b}, or by Hanyga in \cite{Hanyga2020}, a necessary requirement  to ensure that two functions $\phi(t)$ and $\psi(t)$ form a Sonine pair (without moving to the realm of distributions) is for them to have an integrable singularity at the origin. This is further supported by the analysis in \cite{DiethlemGarrappaGiustiStynes2020} where it was shown that operators based on regular kernels can satisfy the fundamental theorem of fractional calculus (\ref{eq:FundamentalTheorem}) 
only if their action is restricted 
to spaces of functions with severe (and somewhat artificial) constraints (see also \cite{Stynes2018}).

A more detailed characterization of Sonine kernels has been investigated in  \cite{Luchko2021_FCAA,Luchko2021_Mathematics,Luchko2021_Symmetry}, where  major attention was devoted to kernels $\phi(t) \in {\mathcal C}_{-1}(0,T]$, i.e. such that $\phi(t) = t^{p-1}\hat{\phi}(t)$ with $t>0$, $p>0$ and $\hat{\phi}(t) \in {\mathcal C}[0,T]$. In the context of Scarpi’s theory, however, characterizing the kernel $\phi_{\alpha}(t)$ as a ${\mathcal C}_{-1}(0,T]$ function appears quite difficult since just its LT $\Phi_{\alpha}(s)$ is known. Here we do not pursue the goal of establishing a complete and general characterization of $\alpha(t)$ leading to kernel pairs that satisfy the Sonine condition; such a hard task is left for future investigations. Instead, here we focus on some minimal arguments that can be employed to grant the viability of our approach in some simplified scenarios. 

Consider a given transition function $\alpha(t)$ for which the LT $A(s)$ exists, then the kernels $\phi_{\alpha}(t)$ and $\psi_{\alpha}(t)$ automatically satisfy the Sonine equation (\ref{eq:SonineEquation}) provided that $\Phi_{\alpha}(s)$ and $\Psi_{\alpha}(s)$ admit real-valued inverse LTs. Indeed,  $\Phi_{\alpha}(s)$ and $\Psi_{\alpha}(s)$ satisfy (\ref{eq:SonineEqationLT}) by construction.

The real-valued character of the inverse LTs $\phi_{\alpha}(t)$ and $\psi_{\alpha}(t)$ of $\Phi_{\alpha}(s)$ and $\Psi_{\alpha}(s)$ is guaranteed by the following simple result.

\begin{prop}
Let $\alpha:[0,T] \to \Rset$ be a function whose LT is $A(s)$. If there exist functions  $\phi_{\alpha}(t)$ and $\psi_{\alpha}(t)$ which are LT-inverse of $\Phi_{\alpha}(s)=s^{sA(s)-1}$ and $\Psi_{\alpha}(s)=s^{-sA(s)}$, then they are real-valued functions.
\end{prop}

\begin{proof}
Let $s^{\star}$ and $g^{\star}(t)$ denote the complex conjugate of a complex variable $s$ and of a  complex-valued function $g(t)$, respectively, and observe that if $G(s)$ is the LT of $g(t)$, then $G^{\star}(s^{\star})$ is the LT of $g^{\star}(t)$. Therefore, to ensure that $\phi_{\alpha}(t)$ and $\psi_{\alpha}(t)$ are real-valued it is sufficient to show that $\Phi_{\alpha}(s) = \Phi^{\star}_{\alpha}(s^{\star})$ and $\Psi_{\alpha}(s) = \Psi^{\star}_{\alpha}(s^{\star})$.

Since $\alpha(t)$ is real, then $\alpha^{\star}(t) = \alpha(t)$ and hence $A^{\star}(s^{\star}) = A(s)$. Setting $G(s) = sA(s)-1$, for which one has that $G^{\star}({s}^\star)=G(s)$, then one finds
\[
\Phi^{\star}_{\alpha}(s^{\star}) 
%= \Bigl(\bar{s}^{G(s^{\star})}\Bigr)^{\star}
= \Bigr(\eu^{G(s^{\star}) \ln s^{\star}} \Bigr)^{\star}
= \eu^{G^{\star}(s^{\star}) \bigl(\ln s^{\star} \bigr)^{\star}}
= \eu^{G(s) \ln s} = s^{G(s)} = \Phi_{\alpha}(s) \, ,
\]
by taking advantage of some elementary properties of complex functions. Similarly, one can show that  
$\Psi_{\alpha}(s) = \Psi^{\star}_{\alpha}(s^{\star})$.
% (namely  $s^{z} = \eu^{z \ln s}$, $\overline{\eu^{z}} = \eu^{\bar{z}}$ and $\overline{\ln(z)} = \ln(\bar{z})$) have been used. 
\end{proof}

To find a necessary condition ensuring that $\Phi_{\alpha}(s)$ and $\Psi_{\alpha}(s)$ are LTs of some functions $\phi_{\alpha}(t)$ and $\psi_{\alpha}(t)$ we observe that if a complex-valued function $G(s)$ is the LT of $g(t)$, then $G(s) \to 0$ as $s\to \infty$. Therefore, if one assumes that $\alpha(t)$ admits a limit in $(0,1)$ as $t \to 0^+$, {\em i.e.},  
\[
\lim_{t \to 0^{+}} \alpha(t) = \bar{\alpha} \in (0,1), 
\]
then the initial value theorem \cite[\S 12.7]{LePage1980} for the LT implies that $sA(s) \to \bar{\alpha} \in (0,1)$ as $s\to \infty$. This ensures that $\Phi_{\alpha}(s)=s^{sA(s)-1} \to 0$ and $\Psi_{\alpha}(s)=s^{-sA(s)} \to 0$ as $s \to \infty$. Therefore one can conclude that any function $\alpha(t):[0,T]\to(0,1)$ admitting a LT is a suitable candidate for generating a pair of Scarpi variable-order operators $\{\DS^{\alpha(t)}_0 , \IS^{\alpha(t)}_0\}$ provided that $\Phi_{\alpha}(s)=s^{sA(s)-1}$ and $\Psi_{\alpha}(s)=s^{-sA(s)}$ are LTs of some functions $\phi_{\alpha}(t)$ and $\psi_{\alpha}(t)$.

Note that, for practical reasons, in this work we further require an explicit analytic expression for $A(s)$.

\subsection{Kochubei's General Fractional Calculus and Scarpi's operators} 

In Kochubei's General Fractional Calculus (GFC) \cite{Kochubei2011,Kochubei2019a}, the operator 
$$
D_{\phi} f(t) = \frac{\du}{\du t} \int _0 ^t  \phi (t-\tau) f(\tau) \du \tau - \phi(t) f(0)
, \quad t \in (0,T] ,
$$
defines a Caputo–Dzhrbashyan type General Fractional Derivative
if the kernel $\phi(t)$ has the following properties:
\begin{description}
\item[A1:] the LT $\Phi(s)$ of $\phi(t)$ exists for all $s>0$; 
\item[A2:] $\Phi(s)$ is a  Stieltjes function, {\em i.e.}, it admits the integral representation 
\begin{equation}\label{eq:spercrep}
    \Phi(s) = 
    \frac{a}{s} + b + 
    \int_0^{\infty} \frac{K_{\alpha}^{\Phi}(r)}{s+r} \du r \, ,
\end{equation}
with $a,b \ge 0$ and $K_{\alpha}^{\Phi}(r)\ge 0$ a (non-negative) spectral distribution;
\item[A3]: $\Phi(s) \to 0$ and $s\Phi(s) \to \infty$ as $s \to \infty$,
\item[A4]: $\Phi(s) \to \infty$ and $s\Phi(s) \to 0$ as $s \to 0$.
\end{description}
Then, it is easy to see that denoting by $\Psi (s) := 1/(s\Phi(s))$ one has that $\psi (t)$ and $\phi (t)$ form a Sonine pair. 

This theory might appear rather appealing for our purposes since it relies completely on the LT representation of the kernel $\phi(t)$. However, the aforementioned conditions further constrain $\phi(t)$ and $\psi(t)$ to be completely monotone (CM) functions. In other words, these conditions guarantee that the solution of the relaxation equation
$$
D_{\phi} f(t) = - \lambda f(t) \, , \,\, f(0^+) = f_0 \, , \,\, \lambda >0 
$$
is CM \cite{Kochubei2011,Kochubei2019a,Mainardi2010}. 

Requiring the solution of a relaxation equation of variable order to be CM is a bit too restrictive in this case and makes GFC hardly applicable to Scarpi's theory.  In fact, as we shall see with some numerical examples in the following section, even very simple transition functions $\alpha(t)$ yield ``derivative kernels'' $\phi_{\alpha}(t)$ that violate {\bf A2}, thus 
supporting the conclusion that Kochubei's GFC is not the proper theoretical framework
for this type of variable-order operators.

\section{Physically relevant examples of transition functions}\label{S:TransitionFunctions}

In this Section we present some examples of variable-order functions $\alpha(t)$. We confine to potentially physically interesting scenarios where $\alpha(t)$ shows a monotone transition from an initial order $\alpha_1$ to a final order $\alpha_2$, where the latter is only reached asymptotically as $t\to \infty$. Various expressions for $\alpha(t)$ are presented here and  for each of them we show the emerging kernels $\phi_{\alpha}(t)$ and $\psi_{\alpha}(t)$ associated to the corresponding $\DS^{\alpha(t)}_0$ and $\IS^{\alpha(t)}_0$ defined respectively in (\ref{eq:ScarpiDerivativeConvlution}) and (\ref{eq:ScarpiIntegralConvolution}).

Even when $\alpha(t)$ and its LT $A(s)$ are given by simple expressions, in general it is not possible to provide an explicit representation of the kernels $\phi_{\alpha}(t)$ and $\psi_{\alpha}(t)$. Therefore, they need to be  evaluated numerically by means of LT inversion of  $\Phi_{\alpha}(s)$ and  $\Psi_{\alpha}(s)$. On the one hand, this complication constituted the main reason why Scarpi's ideas have been overlooked for so long. On the other hand, over the years some very powerful methods for the numerical inversion of the LT have been developed and can be easily exploited in this context. To lighten the presentation we avoid describing here the technical details about the numerical strategy adopted for the numerical inversion of LTs and we confine it to the Appendix \ref{S:NumericalInversionLT}.

\subsection{Example 1: Exponential transition}\label{SS:OT_Exp}

For $0<\alpha_1<\alpha_2 < 1$ and a real constant $c>0$, we consider the function 
$$
\alpha(t) = \alpha_2 + (\alpha_1 - \alpha_2) \eu^{-ct}
$$
describing a variable-order transition from $\alpha_1$ to $\alpha_2$ according to an exponential law with rate $-c$. It is simple to evaluate the LT of $\alpha(t)$ as
\[
    A(s) 
    = \int_{0}^{\infty} \eu^{-st} \alpha(t) \du t 
    = \frac{\alpha_2 c + \alpha_1 s}{s(c+s)}
\]
and, hence, 
\[
    \Phi_{\alpha}(s) = s^{sA(s) - 1} = s^{\frac{(\alpha_2-1)c + (\alpha_1-1)s}{c+s}}
    , \quad
    \Psi_{\alpha}(s) = s^{-s A(s)} = s^{-\frac{\alpha_2 c + \alpha_1 s}{c+s}}.
\]

The spectral distribution $K_{\alpha}^{\Phi}(r)$ that yields the integral representation \eqref{eq:spercrep} of $\Phi_{\alpha}(s)$ can be evaluated by means of the Titchmars \cite{Titchmarsh1986} inversion formula %$K_{\alpha}^{\Phi}(r) = \mp {\rm Im} \bigl[ \Phi_{\alpha}(s) \bigr|_{s=r \eu^{\pm \iu \pi}} \bigr]/\pi$ 
%\[
%K_{\alpha}^{\Phi}(r) = \mp \frac{1}{\pi} {\rm Im} \Bigl[ \Phi_{\alpha}(s) \bigr|_{s=r \eu^{\pm \iu \pi}} \Bigr]
%\]
%that implies
\[
	\begin{aligned}
K_{\alpha}^{\Phi}(r) 
&= \mp \frac{1}{\pi} {\rm Im} \bigl[ \Phi_{\alpha}(s) \bigr|_{s=r \eu^{\pm \iu \pi}} \bigr] \\
&= - \frac{1}{\pi} r^{ \frac{(\alpha_2-1)c - (\alpha_1-1)r}{c-r}} \sin \Bigl[ \frac{(\alpha_2-1)c - (\alpha_1-1)r}{c-r} \pi \Bigr] . 
\end{aligned}
\]

Since $K_{\alpha}^{\Phi}(r) \ge 0$ for $r \ge 0$ is satisfied only if $c=0$ or $\alpha_1=\alpha_2$, namely when the time-dependency of $\alpha(t)$ is suppressed, Kochubei's GFC theory does not apply and $\phi_{\alpha}(t)$ (as well as  the solution of the associated relaxation equation) clearly is not a CM function.

Although it is reasonable to assume, for some physical models, that $\alpha_1$ and $\alpha_2$ are close values, we shall consider distant enough values for these parameters, as illustrated in Figure \ref{fig:Fig_DecayEXP_alpha} for $\alpha_1=0.6$ and $\alpha_2=0.8$, in order to be able to graphically present the asymptotic behaviour of $\phi_{\alpha}(t)$ and $\psi_{\alpha}(t)$ in a nice way.

As one can see from Figure \ref{fig:Fig_DecayEXP_phi_psi}, the resulting kernels $\phi_{\alpha}(t)$ and $\psi_{\alpha}(t)$  start as the corresponding kernels of the standard fractional operators of order $\alpha_1$ and asymptotically converge to the kernels of the operators of order $\alpha_2$. This behaviour can be better appreciated in Figure \ref{fig:Fig_DecayEXP_phi_psi_log} where $\phi_{\alpha}(t)$ and $\psi_{\alpha}(t)$ are plotted in logarithmic scale.

\begin{figure}[tph]
\centering
\begin{tabular}{c}
\includegraphics[width=0.46\textwidth]{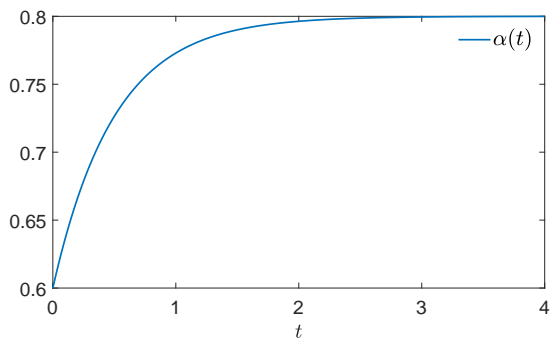} \
\end{tabular}
\caption{Plot of $\alpha(t)$ for variable-order transition of exponential type ($c=2.0$) from $\alpha_1=0.6$ to $\alpha_2=0.8$.} \label{fig:Fig_DecayEXP_alpha}
\end{figure}

\begin{figure}[tph]
\centering
\begin{tabular}{cc}
\includegraphics[width=0.46\textwidth]{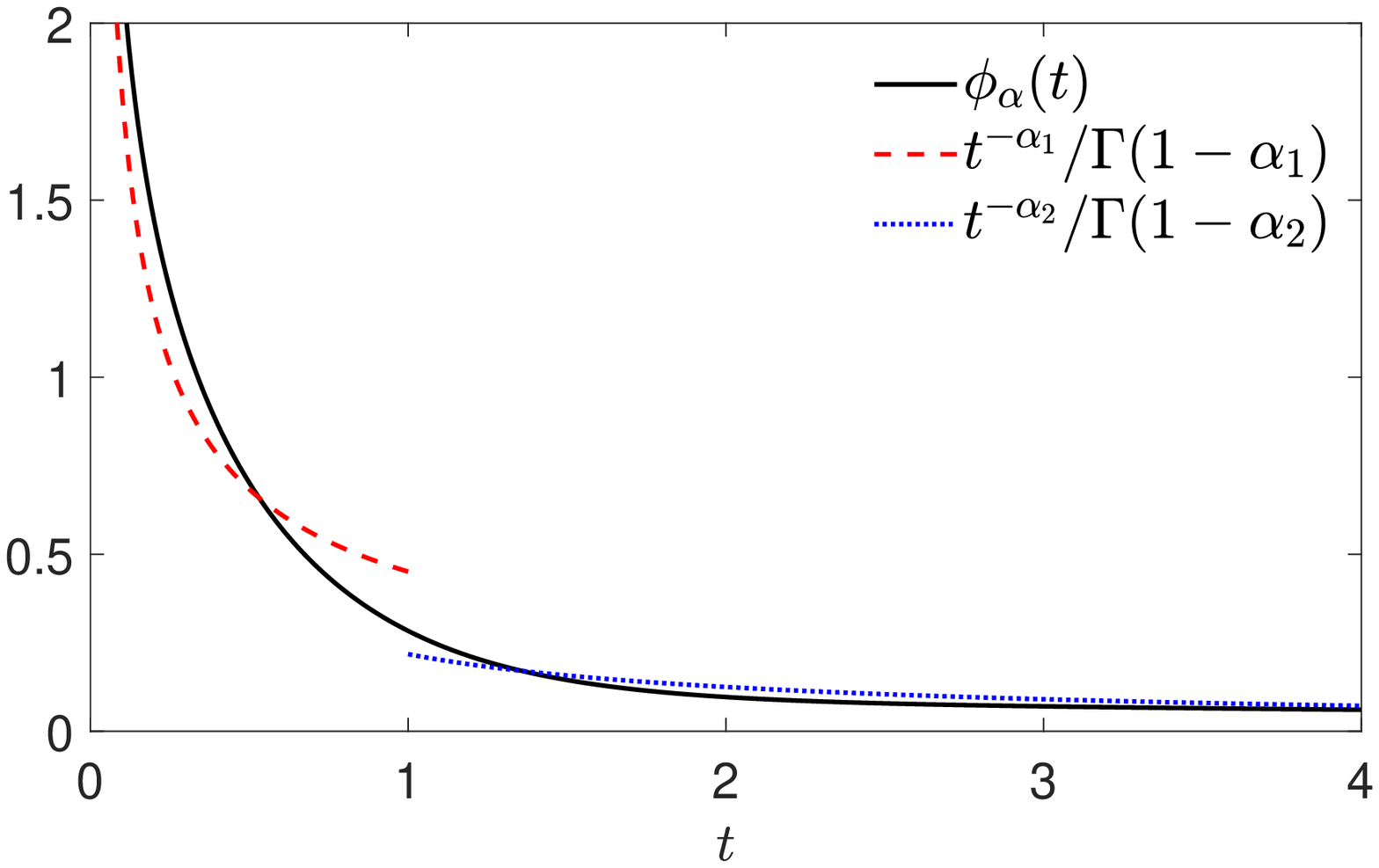} &
\includegraphics[width=0.46\textwidth]{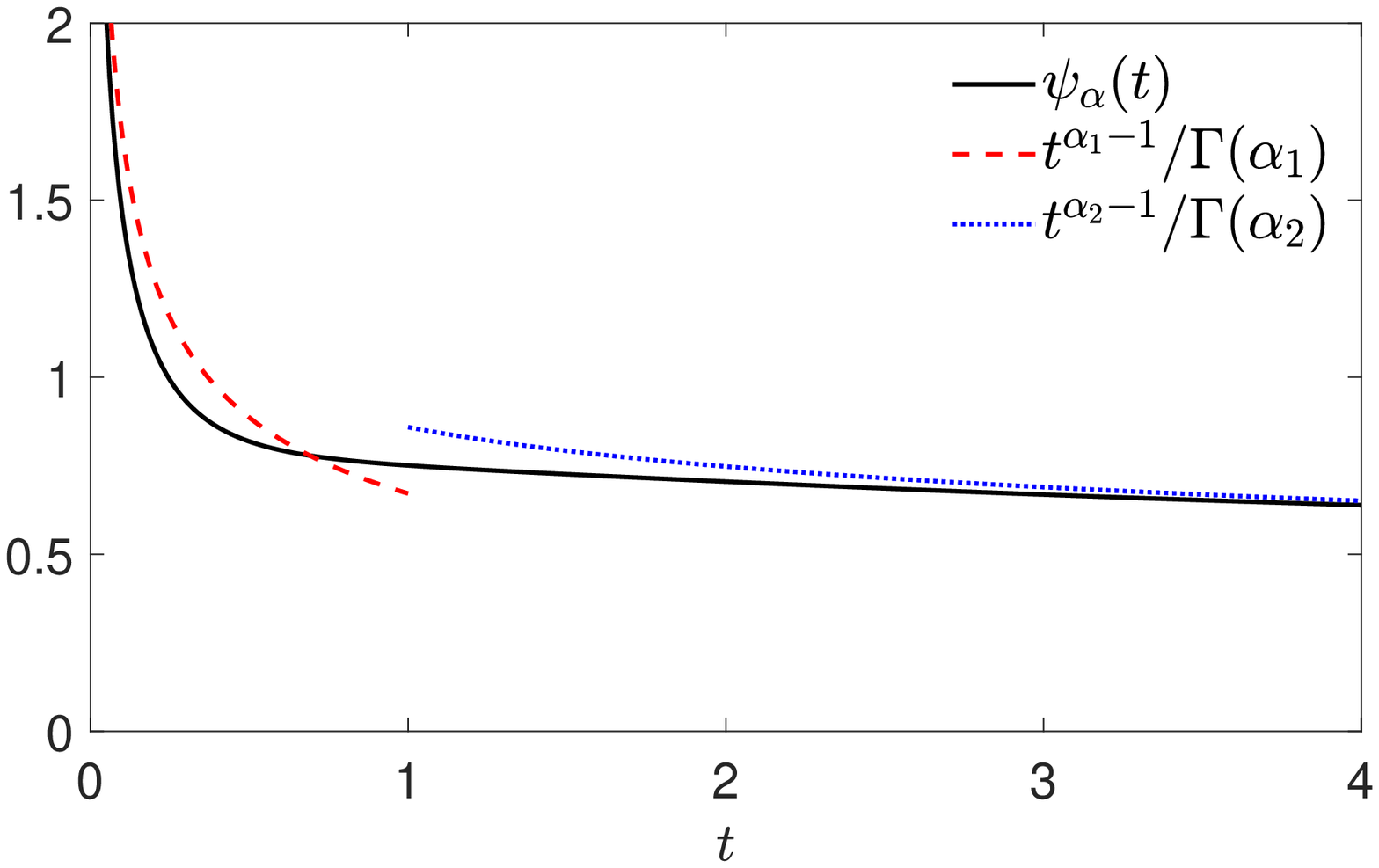} \
\end{tabular}
\caption{Plot of kernels $\phi_{\alpha}(t)$ (left plot) and $\psi_{\alpha}(t)$ (right plot) for variable-order transition of exponential type ($c=2.0$) from $\alpha_1=0.6$ to $\alpha_2=0.8$.} \label{fig:Fig_DecayEXP_phi_psi}
\end{figure}

\begin{figure}[tph]
\centering
\begin{tabular}{cc}
\includegraphics[width=0.46\textwidth]{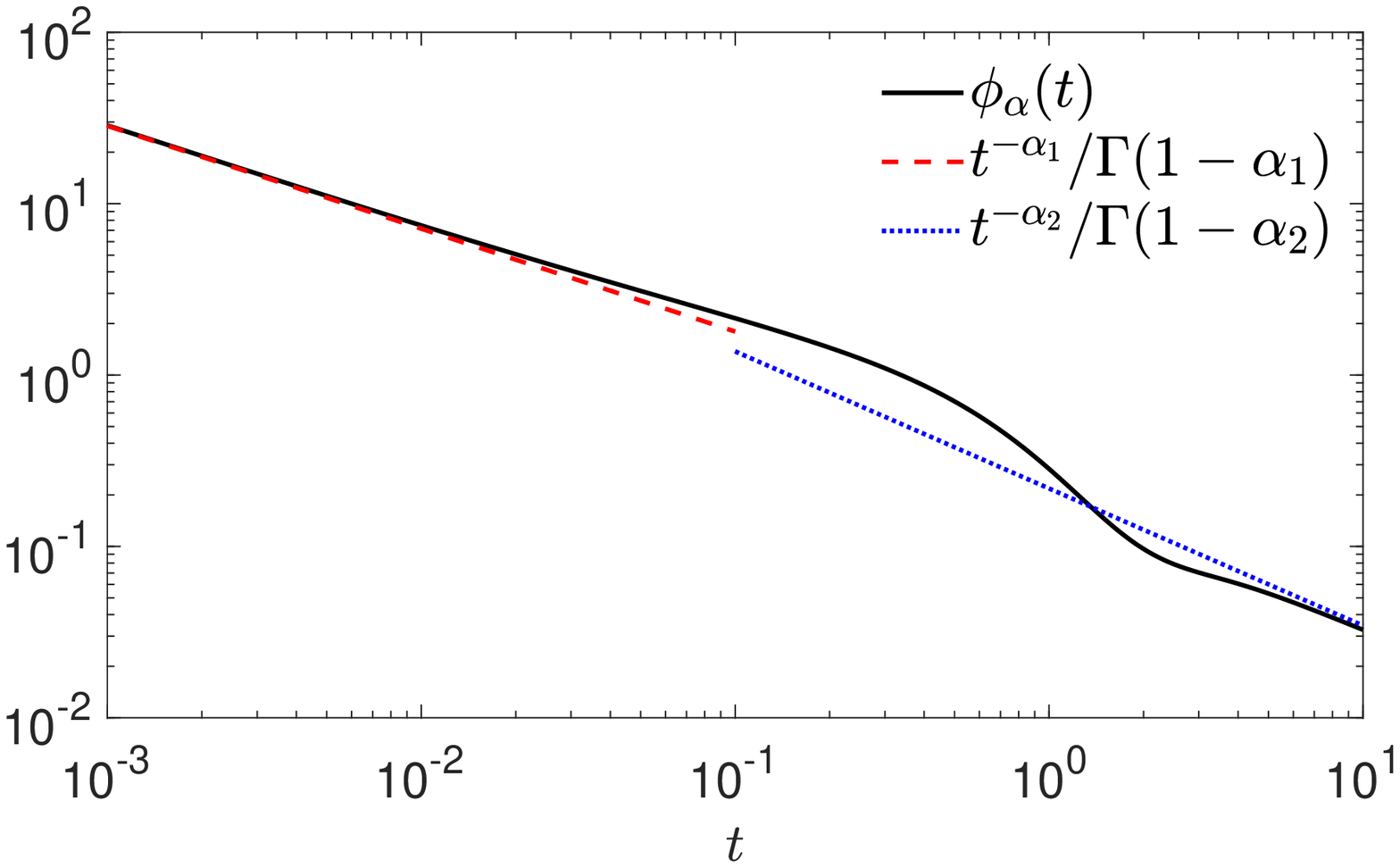} &
\includegraphics[width=0.46\textwidth]{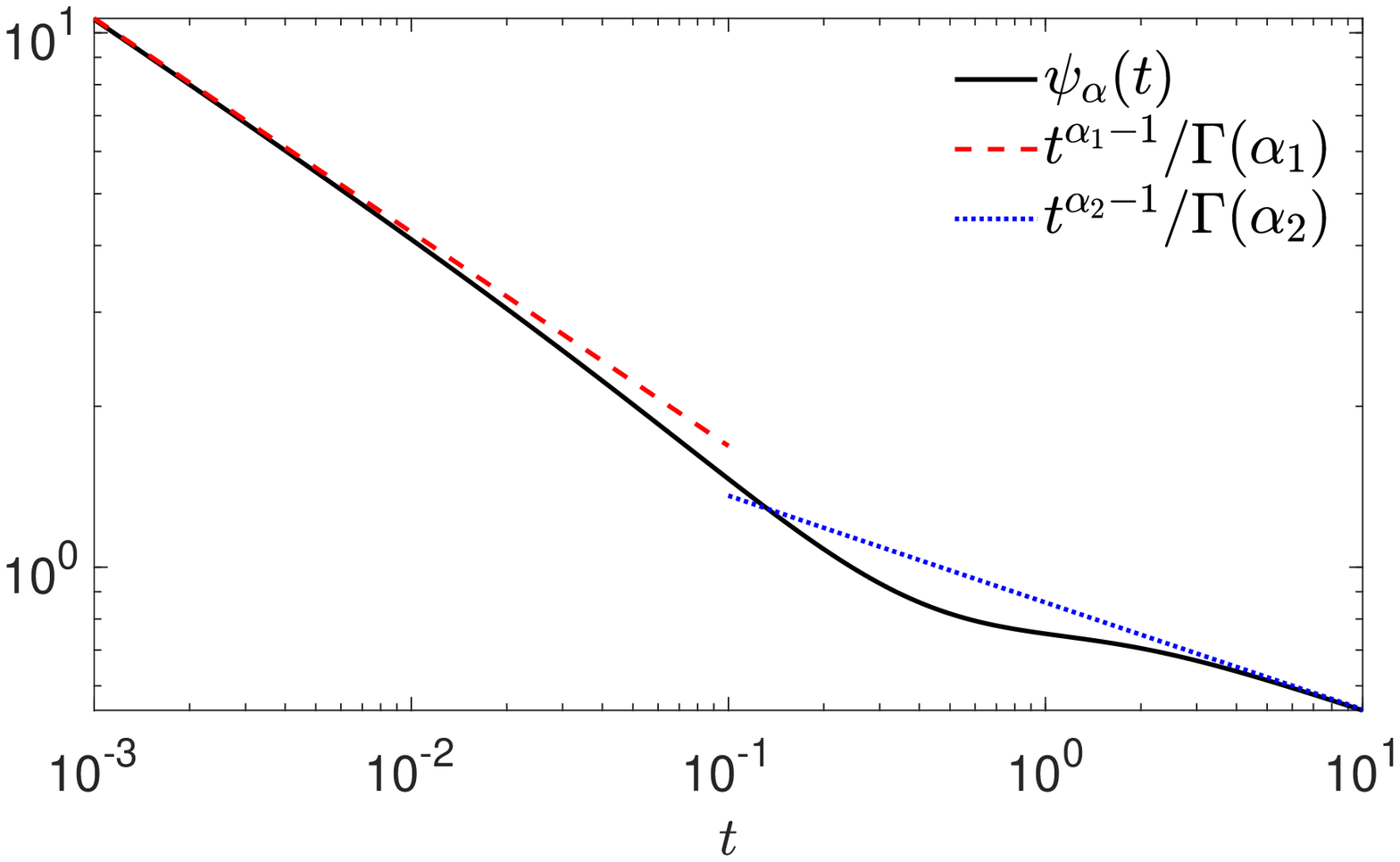} \
\end{tabular}
\caption{Plot of kernels $\phi_{\alpha}(t)$ (left plot) and $\psi_{\alpha}(t)$ (right plot) for variable-order transition of exponential type ($c=2.0$) from $\alpha_1=0.6$ to  $\alpha_2=0.8$ (logarithmic scale).} \label{fig:Fig_DecayEXP_phi_psi_log}
\end{figure}

\subsection{Example 2: Order transition of Mittag-Leffler type} 
The previous example can be generalized by replacing the exponential with the Mittag-Leffler (ML) function, {\em i.e.},
\[
    \alpha(t) = \alpha_2 + (\alpha_1 - \alpha_2) E_{\beta}(-c t^{\beta}) ,
\]
where
\[
    E_{\beta}(z) = \sum_{k=0}^{\infty} \frac{z^k}{\Gamma(\alpha k +\beta)} 
\]
is the one parameter ML function (see, for instance \cite{GorenfloKilbasMainardiRogosin2020}). This procedure gives a better control on
the transition from $\alpha_1$ to $\alpha_2$ thanks to the additional parameter $\beta$.

The representation of this variable-order function $\alpha(t)$ is provided in the left panel of Figure \ref{fig:Fig_DecayMLF_alpha} for $\beta=0.7$ and $c=2.0$.
Clearly, the transition presented in Section \ref{SS:OT_Exp} is just a particular case of the one presented here since $\eu^{x} = E_{1}(x)$.

It is now fairly easy to compute the LT of $\alpha(t)$, that reads \cite{GorenfloKilbasMainardiRogosin2020} 
\[
    A(s) 
    = \int_{0}^{\infty} \eu^{-st} \alpha(t) \du t 
    = \frac{\alpha_2 c + \alpha_1 s^{\beta}}{s(c+s^{\beta})} 
\]
and, hence,
\[
    \Phi_{\alpha}(s) = s^{sA(s) - 1} = s^{\frac{(\alpha_2-1)c + (\alpha_1-1)s^{\beta}}{c+s^{\beta}}}
    , \quad
    \Psi_{\alpha}(s) = s^{-s A(s)} = s^{-\frac{\alpha_2 c + \alpha_1 s^{\beta}}{c+s^{\beta}}} ,
\]
and also in this case the corresponding kernels $\phi_{\alpha}(t)$ and $\psi_{\alpha}(t)$ match the kernels of the standard fractional operators of order $\alpha_1$ and $\alpha_2$ in the limitig cases of the model, as shown in  Figures \ref{fig:Fig_DecayMLF_phi_psi} and \ref{fig:Fig_DecayMLF_phi_psi_log}. 

Note that the parameter $\beta$, similarly to the parameter $c$ in the previous case, alters the way in which this transition happens without affecting the initial and final values of the order.

\begin{figure}[tph]
\centering
\begin{tabular}{c}
\includegraphics[width=0.46\textwidth]{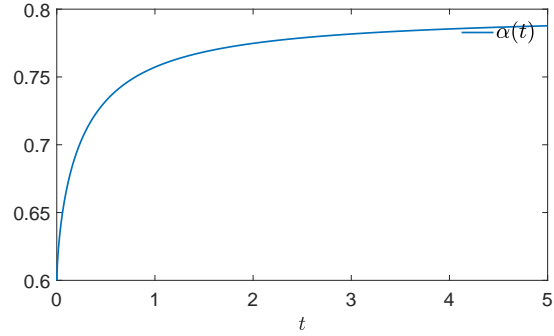}  \
\end{tabular}
\caption{Plot of $\alpha(t)$ for order transition of ML type ($c=2.0$ and $\beta=0.7$) from $\alpha_1=0.6$ to $\alpha_2=0.8$.} \label{fig:Fig_DecayMLF_alpha}
\end{figure}

\begin{figure}[tph]
\centering
\begin{tabular}{cc}
\includegraphics[width=0.46\textwidth]{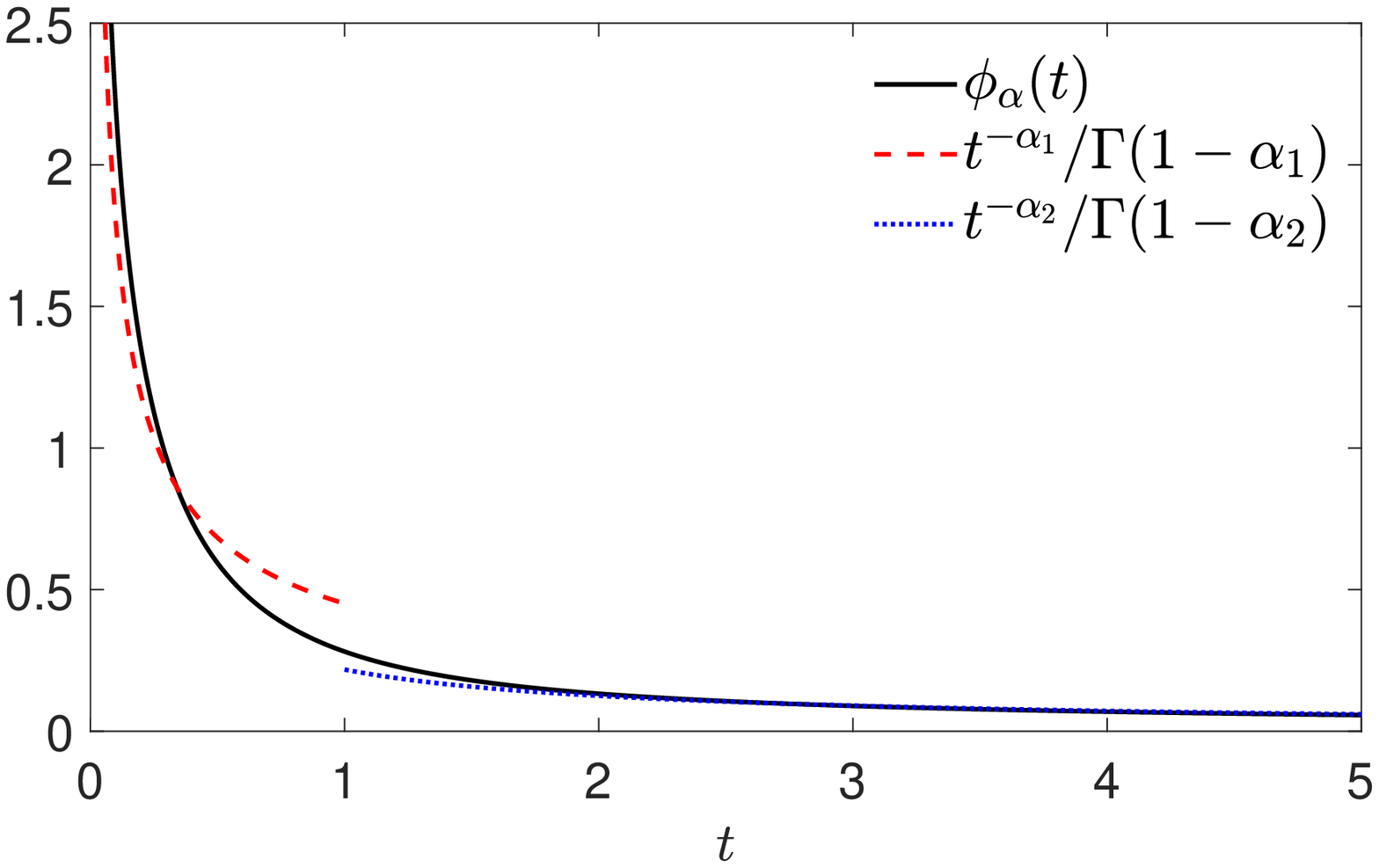} &
\includegraphics[width=0.46\textwidth]{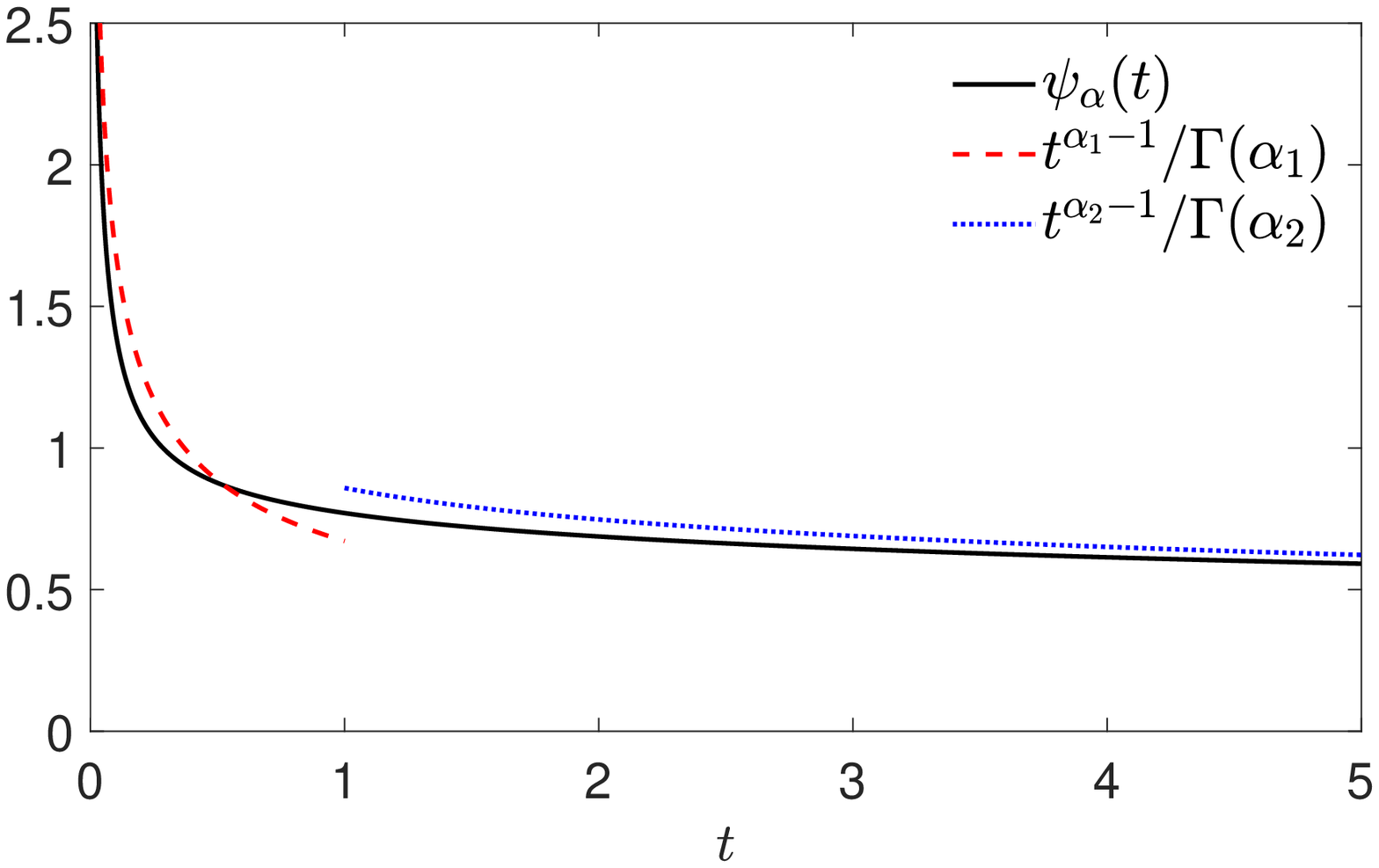} \
\end{tabular}
\caption{Plot of kernels $\phi_{\alpha}(t)$ (left plot) and $\psi_{\alpha}(t)$ (right plot) for order transition of ML type ($c=2.0$ and $\beta=0.7$) from $\alpha_1=0.6$ to $\alpha_2=0.8$.} \label{fig:Fig_DecayMLF_phi_psi}
\end{figure}

\begin{figure}[tph]
\centering
\begin{tabular}{cc}
\includegraphics[width=0.46\textwidth]{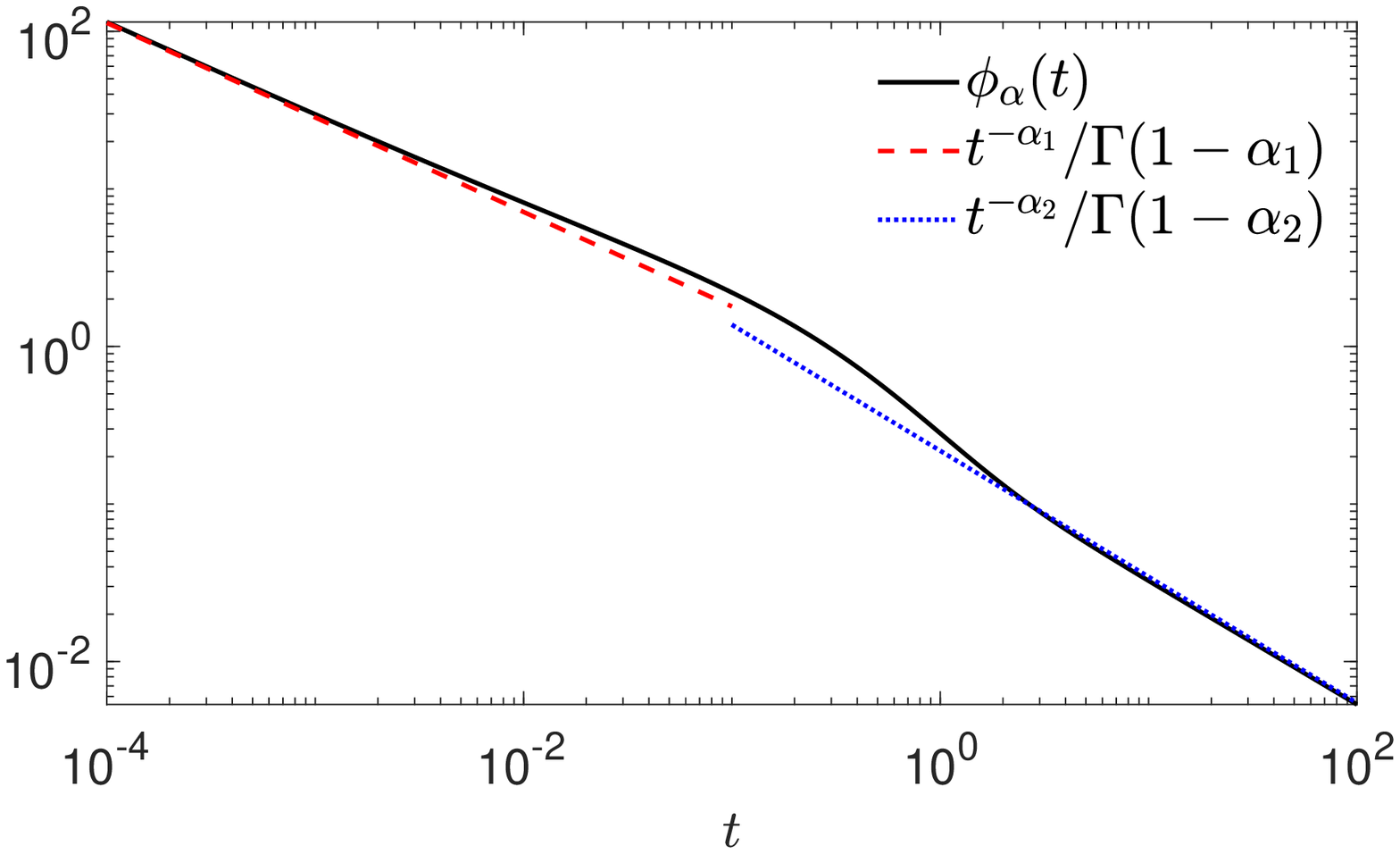} &
\includegraphics[width=0.46\textwidth]{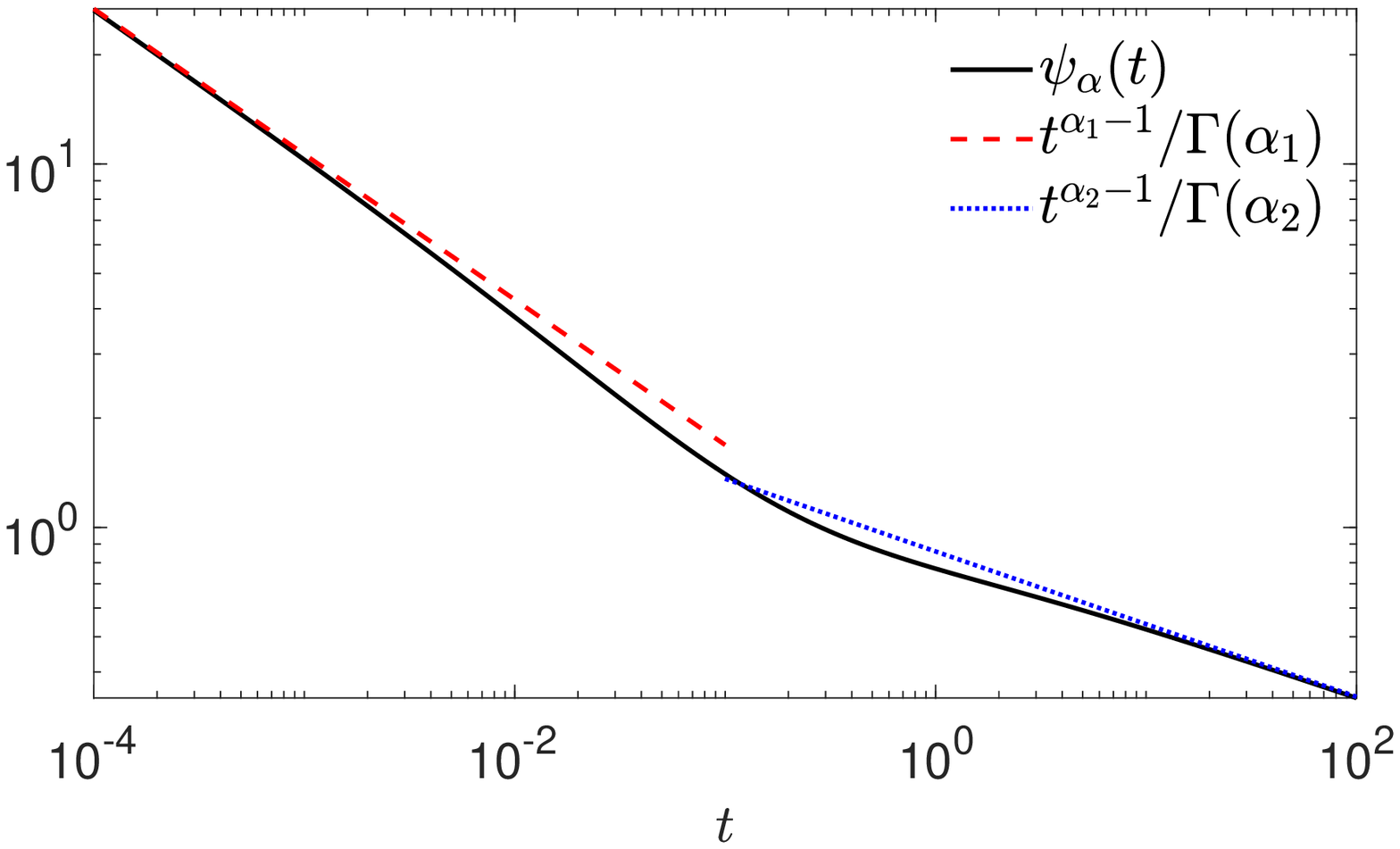} \
\end{tabular}
\caption{Plot of kernels $\phi_{\alpha}(t)$ (left plot) and $\psi_{\alpha}(t)$ (right plot) for order transition of ML type ($c=2.0$ and $\beta=0.7$) from $\alpha_1=0.6$ to $\alpha_2=0.8$ (logarithmic scale).} \label{fig:Fig_DecayMLF_phi_psi_log}
\end{figure}

In Figure \ref{fig:Fig_DecayMLF_alt_phit_b} we compare the behaviour of $\alpha(t)$ and $\psi_{\alpha}(t)$ for the decay of ML-type as we vary the parameter $\beta$. Observe that the case $\beta=1.0$ corresponds to the exponential decay case, as anticipated.

\begin{figure}[tph]
\centering
\begin{tabular}{cc}
\includegraphics[width=0.46\textwidth]{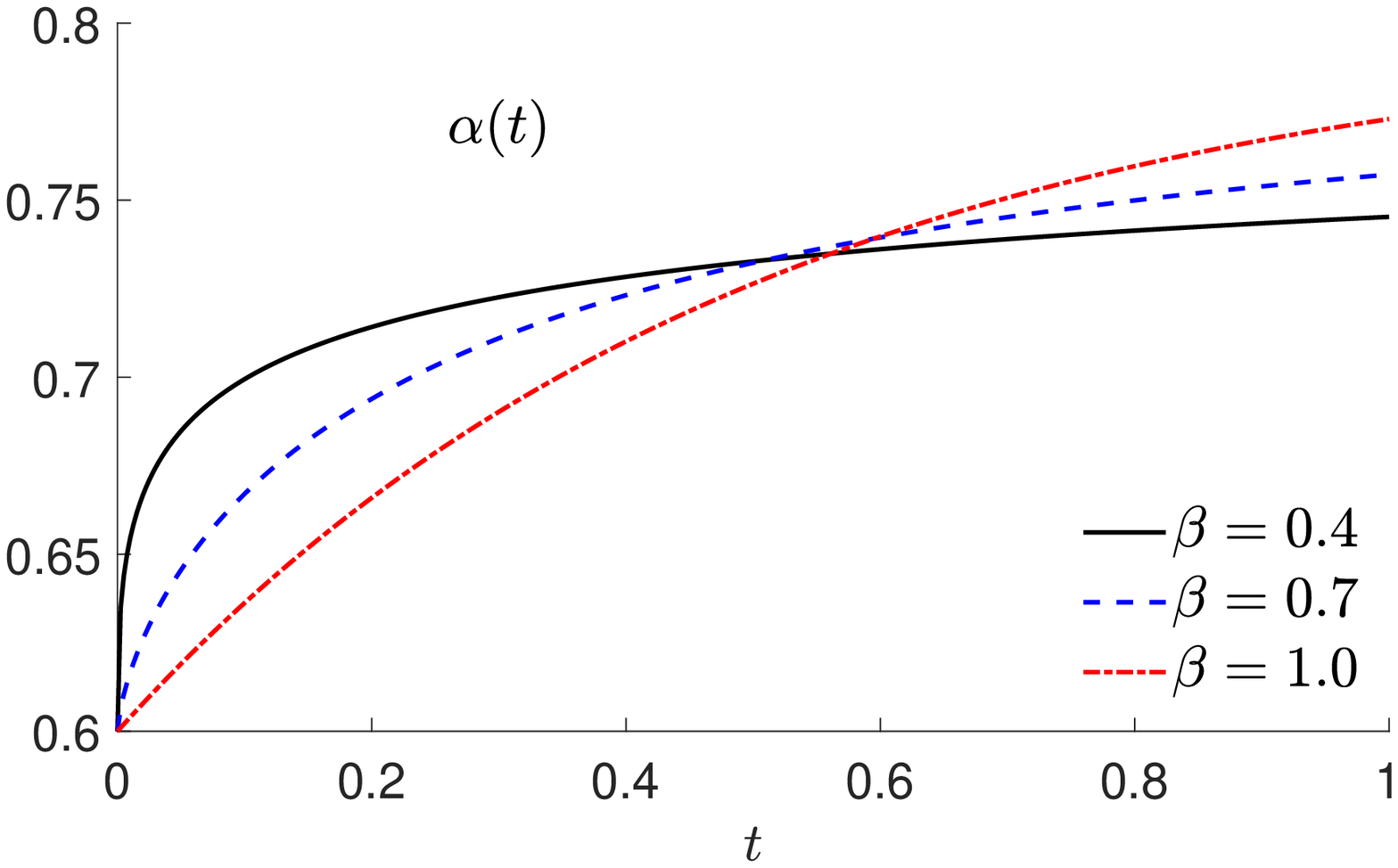} &
\includegraphics[width=0.46\textwidth]{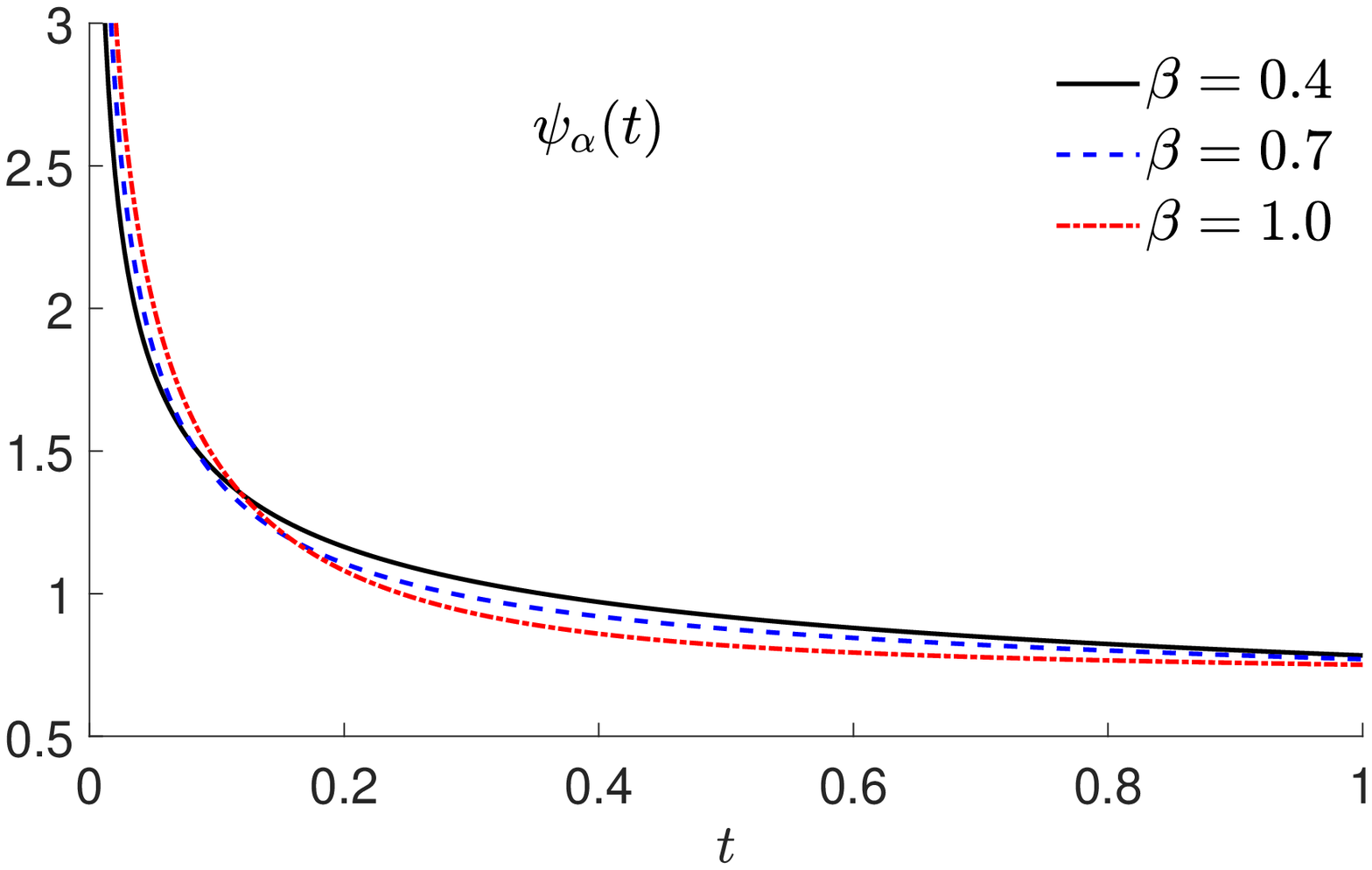} \
\end{tabular}
\caption{Plot of $\alpha(t)$ (left plot) and $\psi_{\alpha}(t)$ (right plot) for the variable order with decay of ML type with $\alpha_1=0.6$, $\alpha_2=0.8$, $c=2.0$ and different values of $\beta$.} \label{fig:Fig_DecayMLF_alt_phit_b}
\end{figure}

\subsection{Example 3: Order transition of {\rm erf} type}

Consider now for $0 < \alpha_1<\alpha_2 < 1$ and $c>0$ the function
\[
    \alpha(t) = \alpha_1 + (\alpha_2-\alpha_1) \erf(\sqrt{ct})
\]
representing a variable order which rapidly increases from $\alpha_1$ to $\alpha_2$ as shown in Figure \ref{fig:Fig_DecayERF_alpha}. Observe that this function can be considered, in some sense, as a further generalization of the variable-order function $\alpha(t)$ based on the ML function since
\[
\erf(\sqrt{ct}) = \sqrt{c}t^{\frac{1}{2}} E_{1,\frac{3}{2}}^{\frac{1}{2}}(-tc)
\]
with $E_{\alpha,\beta}^{\gamma}(z)$ the three-parameter ML function, also known as Prabhakar function (see, {\em e.g.}, \cite{GarrraGarrappa2018,GiustiColombaroGarraGarrappaPolitoPopolizioMainardi2020,GorenfloKilbasMainardiRogosin2020,Prabhakar1971}).

The Laplace transform of $\alpha(t)$ is
\[
    A(s) = \frac{\alpha_1}{s} + (\alpha_2-\alpha_1)\frac{\sqrt{c}}{s \sqrt{s+c}} 
    = \frac{\alpha_2 \sqrt{c} + \alpha_1 \bigl( \sqrt{s+c}-\sqrt{c}\bigr) }{s \sqrt{s+c}} 
\]
and the corresponding function $\phi_{\alpha}(t)$ and $\psi_{\alpha}(t)$ are depicted in Figures \ref{fig:Fig_DecayERF_alt_phi_psi} and \ref{fig:Fig_DecayERF_alt_phi_psi_log}.

\begin{figure}[tph]
\centering
\begin{tabular}{cc}
\includegraphics[width=0.46\textwidth]{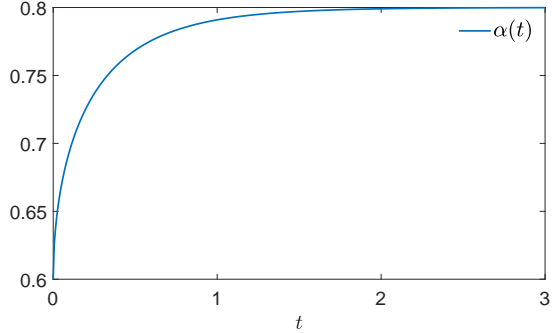}  \
\end{tabular}
\caption{Plot of $\alpha(t)$ for order transition of $\erf$ type ($c=2.0$) from  $\alpha_1=0.6$ to  $\alpha_2=0.8$.} \label{fig:Fig_DecayERF_alpha}
\end{figure}

\begin{figure}[tph]
\centering
\begin{tabular}{cc}
\includegraphics[width=0.46\textwidth]{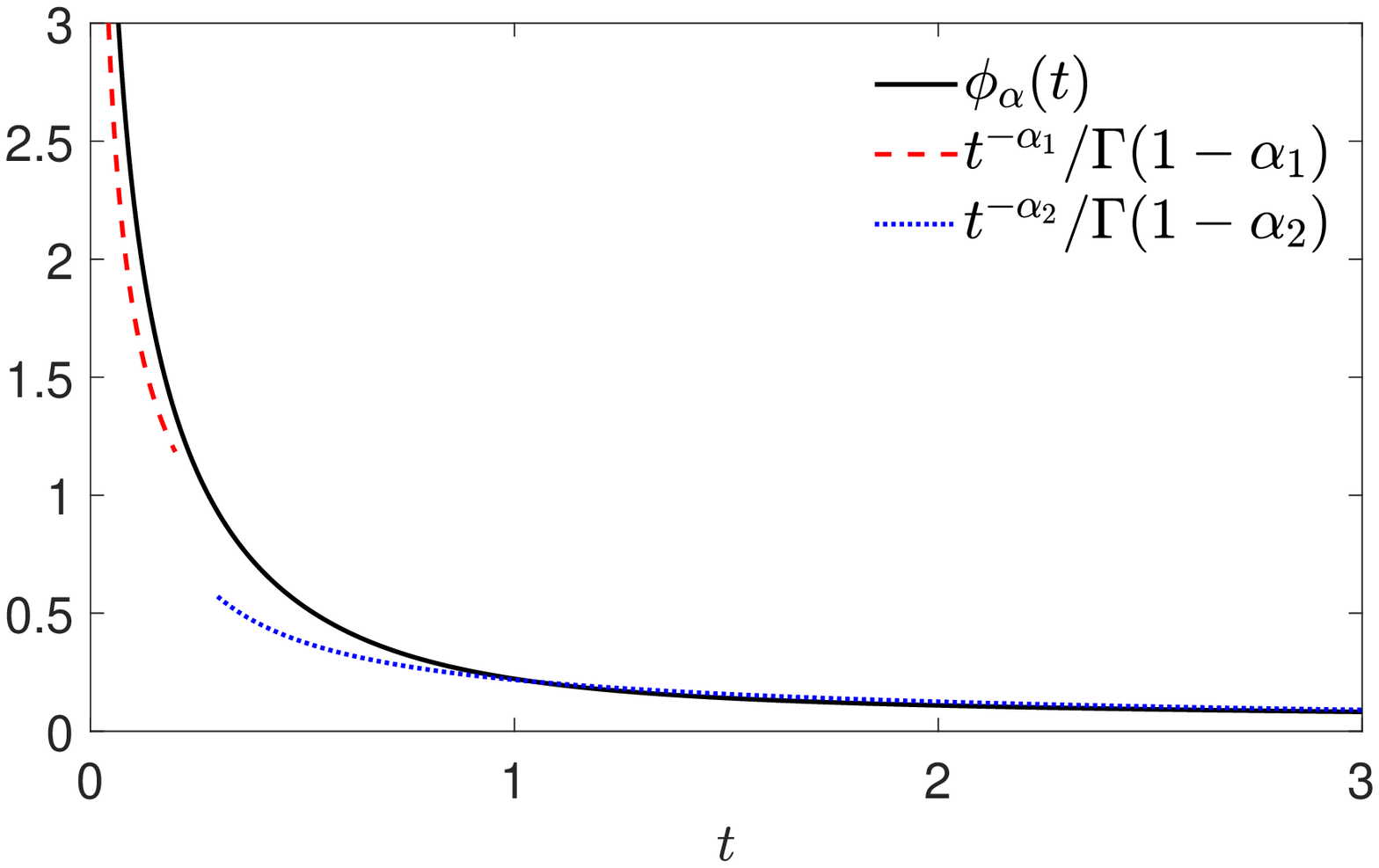} &
\includegraphics[width=0.46\textwidth]{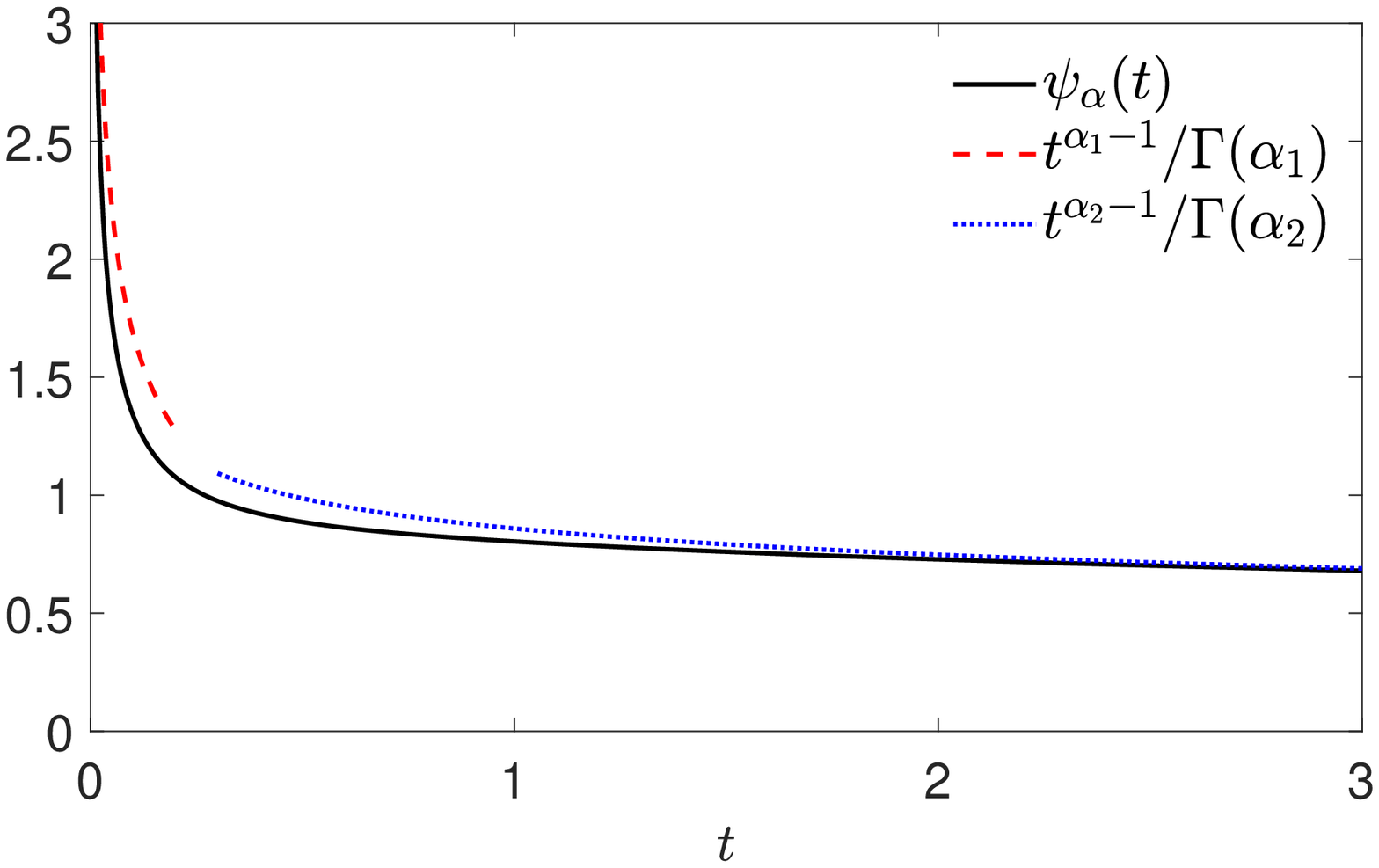} \
\end{tabular}
\caption{Plot of kernels $\phi_{\alpha}(t)$ (left plot) and $\psi_{\alpha}(t)$ (right plot) for order transition of $\erf$ type ($c=2.0$) from  $\alpha_1=0.6$ to $\alpha_2=0.8$.} \label{fig:Fig_DecayERF_alt_phi_psi} 
\end{figure}

\begin{figure}[tph]
\centering
\begin{tabular}{cc}
\includegraphics[width=0.46\textwidth]{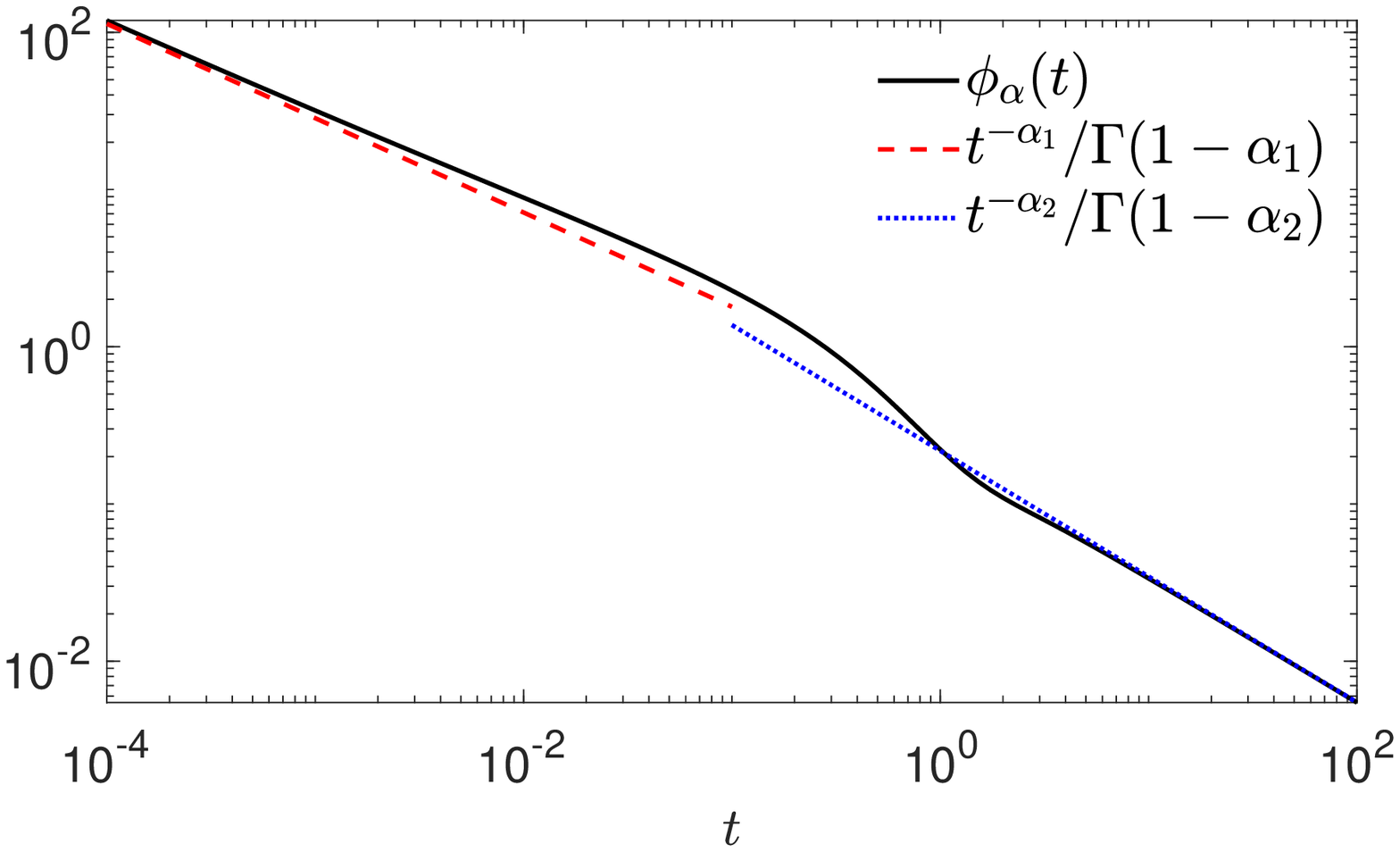} &
\includegraphics[width=0.46\textwidth]{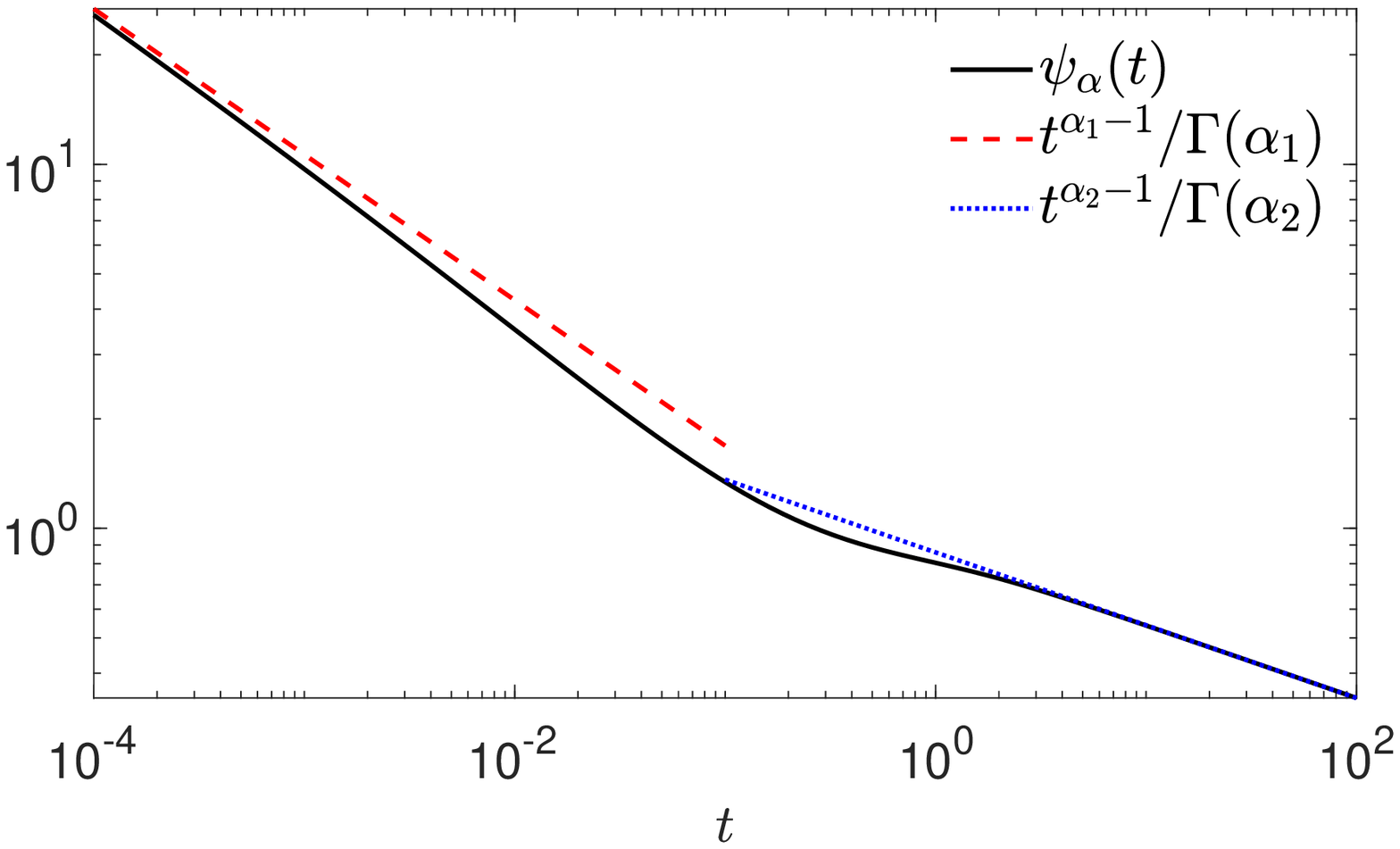} \
\end{tabular}
\caption{Plot of kernels $\phi_{\alpha}(t)$ (left plot) and $\psi_{\alpha}(t)$ (right plot) for order transition of $\erf$ type ($c=2.0$) from  $\alpha_1=0.6$ to $\alpha_2=0.8$ (logarithmic scale).} \label{fig:Fig_DecayERF_alt_phi_psi_log} 
\end{figure}

\section{Fractional relaxation equation with Scarpi derivative}\label{S:S_Relaxation}

The aim of this Section is to provide a preliminary investigation of the variable-order fractional relaxation equation 
\begin{equation}\label{eq:S_Relaxation}
    \left\{ \begin{array}{l}
    \DS^{\alpha(t)}_0 y(t) = - \lambda y(t) \\
    y(0) = y_0 \\
    \end{array} \right. ,
\end{equation}
where $\DS^{\alpha(t)}_0$ is the Scarpi variable-order fractional derivative introduced in Definition \ref{defn:ScarpiDerivative} and  $\lambda > 0$ a real parameter. 

Finding analytical solutions for the initial value problem (\ref{eq:S_Relaxation}) does not seem in general possible since the absence of an explicit representation of the kernel $\phi_{\alpha}(t)$ of $\DS^{\alpha(t)}_0$. Therefore, tackling this problem from a numerical perspective becomes unavoidable and necessary.

%One of the most suitable methods  to discretize  convolution integrals of the type in (\ref{eq:SFDE_VIE}) appears to the be convolution quadrature rules devised and studied by Lubich in a series of pioneering papers \cite{Lubich1988a,Lubich1988b,Lubich2004}. These rules have the great advantage of providing accurate approximations of convolution integrals like the one in  (\ref{eq:SFDE_VIE}) for which the kernel $\phi(t)$ is not known but it is instead known ist  LT $\Phi(s)$ as it is for the Scarpi integral. We think that for general differential equations with the Scarpi derivative these rules are the most promising.

Since the linear nature of (\ref{eq:S_Relaxation}), a simple approach consists in exploiting the LT and its numerical inversion. Indeed,  by applying the LT to both sides of (\ref{eq:S_Relaxation}), and recalling Eq. (\ref{eq:ScarpiDerDefinitionLT}), one finds
\[
 s^{sA(s)} Y(s) - s^{sA(s)-1} y_0 = -\lambda Y(s) ,
\]
with $Y(s)$ the LT of the solution $y(t)$. Therefore, an algebraic manipulation leads to 
\begin{equation}\label{eq:LT_Relaxation}
    Y(s) = \frac{y_0}{s \bigl( 1 + \lambda \Psi_{\alpha}(s) \bigr)} 
\end{equation}
and hence it is possible to evaluate the solution $y(t) = {\mathcal L}^{-1} \bigl( Y(s) \, ; \, t \bigr)$ in the time domain by applying again one of the methods for the numerical inversion of the LT as the one described in the \ref{S:NumericalInversionLT}.

To this end we present the solutions $y(t)$ of the relaxation equation (\ref{eq:S_Relaxation}) with the other transition functions $\alpha(t)$ introduced in Section \ref{S:TransitionFunctions}. In the various plots, together with the solution $y(t)$, we also offer a comparison of $y(t)$ with the solutions $y_{1}(t)$ and $y_2(t)$ of the same relaxation equation with the standard Caputo derivative of order $\alpha_1$ and $\alpha_2$, respectively.

In the first case, see Figure \ref{fig:Fig_Sol_EXP}, the exponential transition $\alpha(t) = \alpha_2 + (\alpha_1 - \alpha_2) \eu^{-ct}$ (with $\alpha_1=0.6$, $\alpha_2=0.8$ and $c=2$) is considered. 

\begin{figure}[tph]
\centering
\includegraphics[width=0.65\textwidth]{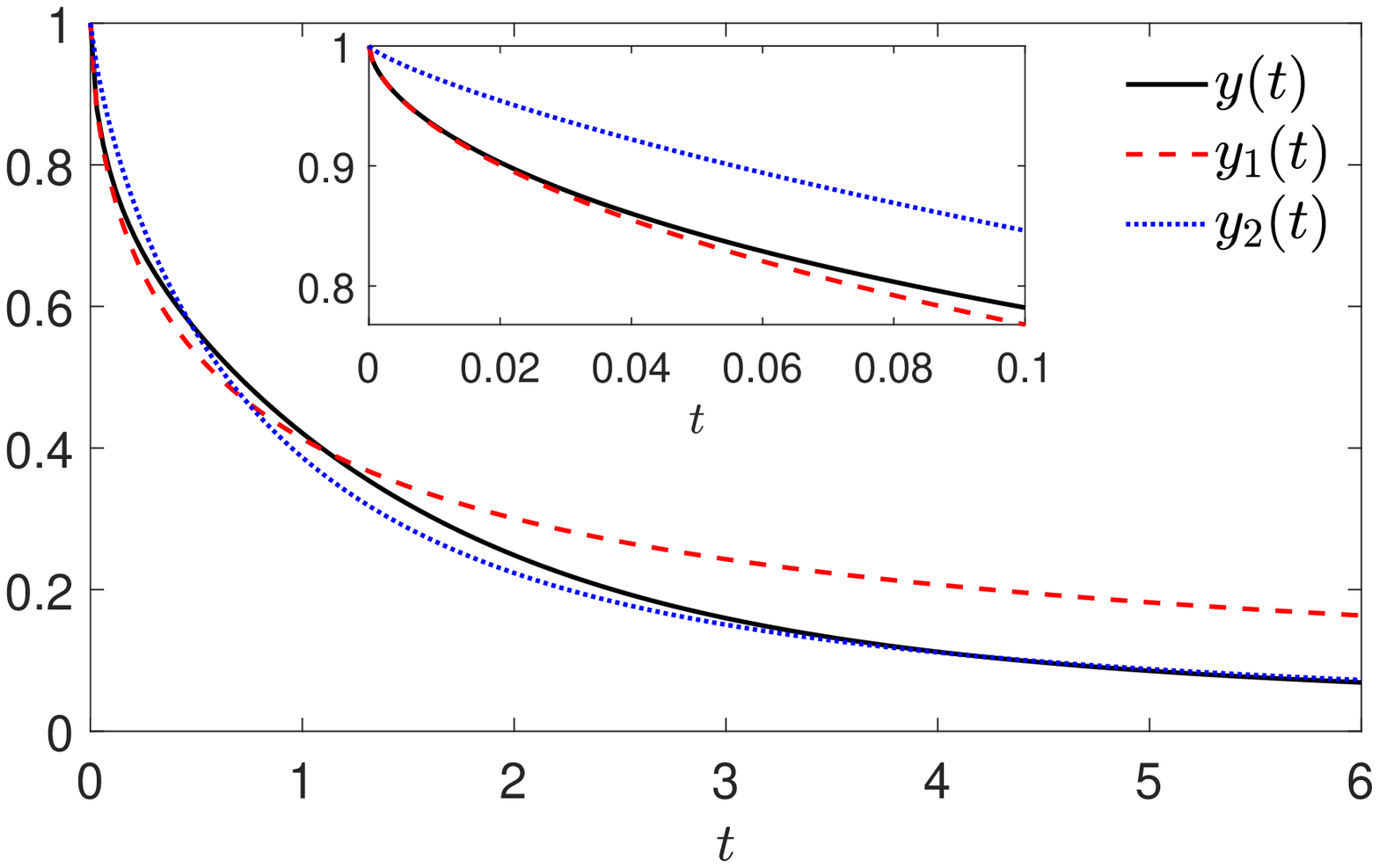} 
\caption{Plot of the solution $y(t)$ of the relaxation equation (\ref{eq:S_Relaxation}), with $\lambda=1$ and $y_0=1$, for  variable-order transition  $\alpha(t) = \alpha_2 + (\alpha_1 - \alpha_2) \eu^{-ct}$, with $\alpha_1=0.6$, $\alpha_2=0.8$ and $c=2$, and comparison with solutions $y_1(t)$ and $y_2(t)$ of the standard fractional relaxation  equations of order $\alpha_1$ and $\alpha_2$.} \label{fig:Fig_Sol_EXP}
\end{figure}

The numerical results show how well the solution with the Scarpi derivative matches the solution of the Caputo relaxation equation of order $\alpha_1$ close to the origin and of the Caputo relaxation equation of order $\alpha_2$ for large $t$. The box in each figure offers a closer look of the solutions near to the origin.

Similar results are obtained with the transition function $\alpha(t) = \alpha_2 + (\alpha_1 - \alpha_2) E_{\beta}(-c t^{\beta})$ (with $\alpha_1=0.6$, $\alpha_2=0.8$, $c=2.0$ and $\beta=0.7$) shown in Figure \ref{fig:Fig_Sol_MLF}, as well as with the transition function $\alpha(t) = \alpha_1 + (\alpha_2-\alpha_1) \erf(\sqrt{ct})$ (with $\alpha_1=0.6$, $\alpha_2=0.8$, $c=2.0$) depicted in Figure \ref{fig:Fig_Sol_ERF}.

\begin{figure}[tph]
\centering
\includegraphics[width=0.65\textwidth]{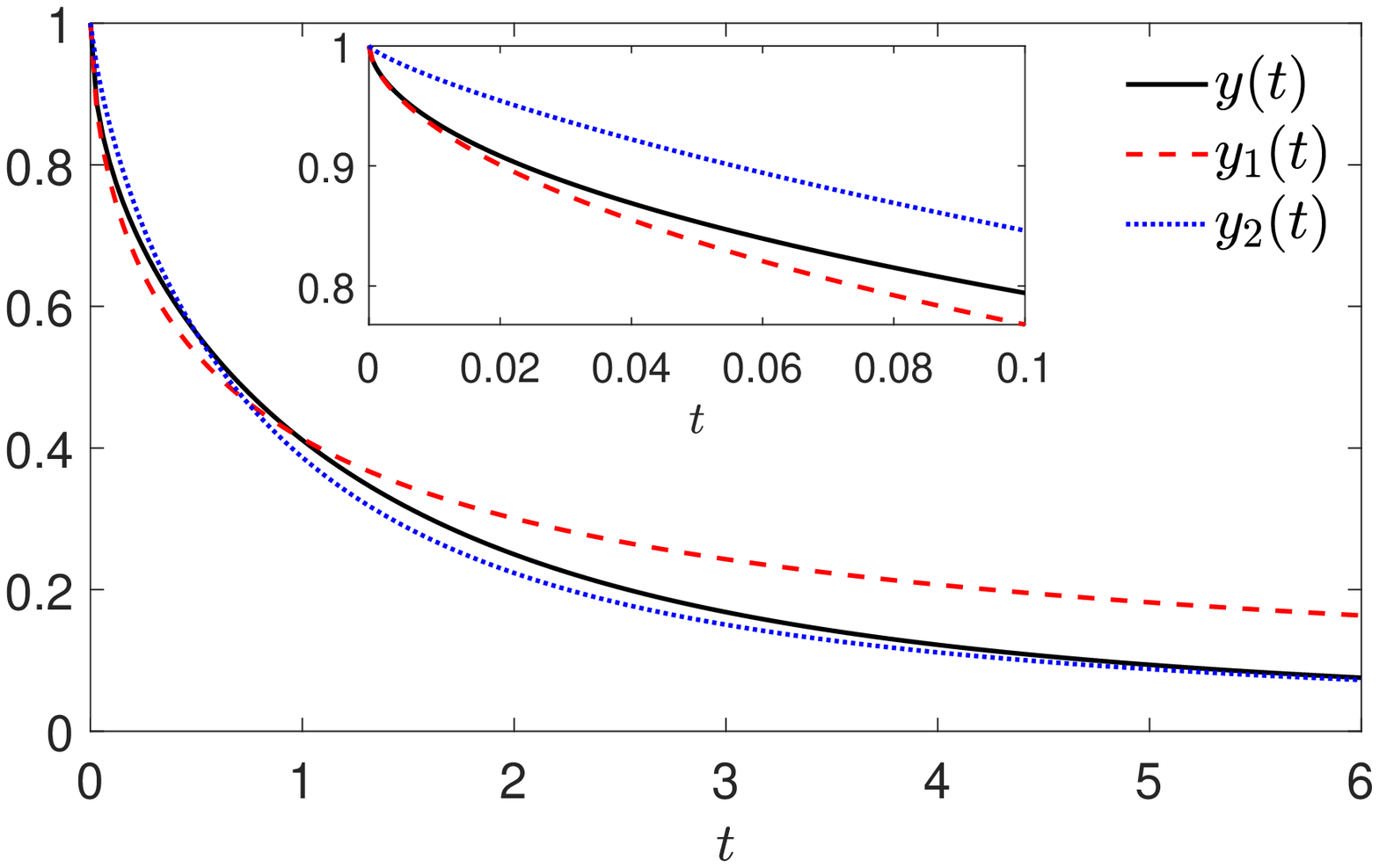} 
\caption{Plot of the solution $y(t)$ of the relaxation equation (\ref{eq:S_Relaxation}), with $\lambda=1$ and $y_0=1$, for variable-order transition $\alpha(t) = \alpha_2 + (\alpha_1 - \alpha_2) E_{\beta}(-c t^{\beta})$, with $\alpha_1=0.6$, $\alpha_2=0.8$, $c=2.0$ and $\beta=0.7$, and comparison with solutions $y_1(t)$ and $y_2(t)$ of the standard fractional relaxation  equations of order $\alpha_1$ and $\alpha_2$.} \label{fig:Fig_Sol_MLF}
\end{figure}

\begin{figure}[tph]
\centering
\includegraphics[width=0.65\textwidth]{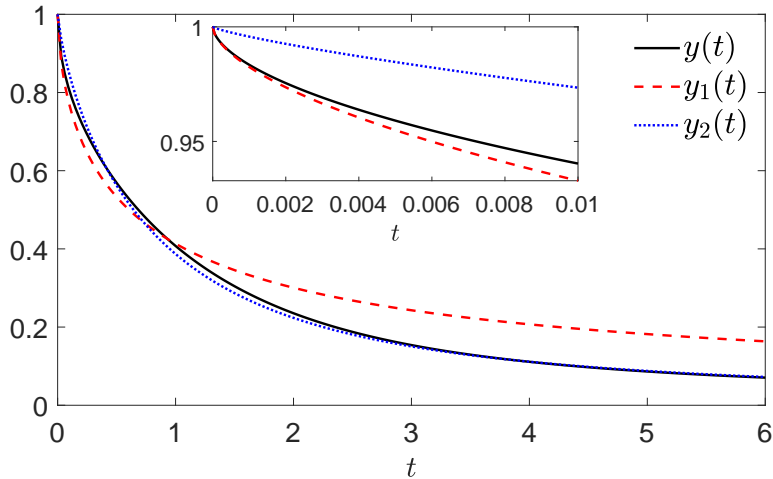}
\caption{Plot of the solution $y(t)$ of the relaxation equation (\ref{eq:S_Relaxation}), with $\lambda=1$ and $y_0=1$, for variable-order transition $\alpha(t) = \alpha_1 + (\alpha_2-\alpha_1) \erf(\sqrt{ct})$, with $\alpha_1=0.6$, $\alpha_2=0.8$, $c=2.0$, and comparison with solutions $y_1(t)$ and $y_2(t)$ of the standard fractional relaxation  equations of order $\alpha_1$ and $\alpha_2$.} \label{fig:Fig_Sol_ERF} 
\end{figure}

Alternatively, one can solve the initial value problem in \eqref{eq:S_Relaxation} by using 
the integral formulation of the problem 
\begin{equation}\label{eq:SFDE_VIE}
    y(t) = y_0 - \lambda \IS^{\alpha(t)}_0  y(t) = y_0 - \lambda \int_0^t \psi_{\alpha}(t-\tau) y(\tau)  \du \tau ,
\end{equation}
and then apply the convolution quadrature rules devised and studied by Lubich in \cite{Lubich1988a,Lubich1988b}. These rules have the great advantage of providing accurate approximations of convolution integrals like the one in (\ref{eq:SFDE_VIE}) for which the kernel $\phi(t)$ is known only through its LT $\Phi(s)$, as it is for the Scarpi integral. Hence, this scheme looks rather promising for handling general fractional differential equations, in special way of nonlinear type, involving the Scarpi derivative.

\begin{remark}
We have confined our discussion to relaxation equations (namely, when $\lambda >0$) but studying the effect of variable-order operators on growth equations ({\em i.e.}, $\lambda <0$) can be of interest, especially for applications to growth models with memory in macroeconomics  \cite{TarasovTarasova2019_Mathematics,Tarasov2020_Mathematics,TarasovTarasova2021}. An extension of the theory of GFC to growth equations is discussed in   
\cite{KochubeiKondratiev2019}. The general theory developed here clearly applies to growth equations as well. However, numerical difficulties may arise in the inversion of the LT due to singularities in (\ref{eq:LT_Relaxation}) when $\lambda <0$.
\end{remark}

\section{Higher-order operators}\label{S:HigherOrderOperators}
Up to this point the presented analysis has been confined to derivatives and integrals of order $0<\alpha(t)<1$. Defining variable-order operators with transition functions $\alpha(t)$ with values spanning a wider range requires some care. Here we shall explore some preliminary ideas in this direction.

Consider Example 1 from Section \ref{S:TransitionFunctions} with the exponential transition function 
\[
\alpha(t) = \alpha_2 + (\alpha_1 - \alpha_2) \eu^{-ct}, \quad c > 0 ,
\]
and where now, for some integer $n \ge 1$, we assume $n-1<\alpha_1<\alpha_2< n$. By following the same reasoning presented in Section \ref{S:NecessaryAssumptions}, we observe that assuming  $\Phi_{\alpha}(s)=s^{sA(s)-1}$ leads now to $\Phi_{\alpha}(s) \to \infty$ as $s\to \infty$ when $n \ge 2$ and hence $\Phi_{\alpha}(s)$ cannot be the LT of any function $\phi_{\alpha}(t)$.

Therefore, 
one has to consider an alternative form of $\Phi_{\alpha}(s)$, when  $n \ge 2$, so that $\phi_{\alpha}(t)$ and $\psi_{\alpha}(t)$ exist and form a Sonine pair. Yet again, the theory in \cite{Luchko2021_Symmetry,Luchko2021_Mathematics} can provide some guidance.

Let $\alpha(t):[0,T] \to (n-1,n)$, $n \in \Nset$, and consider the integral $\IS^{\alpha(t)}_0$ introduced in (\ref{eq:ScarpiIntegralConvolution}). In order to find a derivative $\DS^{\alpha(t)}_0$ acting as the left-inverse of $\IS^{\alpha(t)}_0$ one has to build a kernel $\phi_{\alpha,n}(t)$ satisfying a generalized Sonine equation \cite[Eq. (35)]{Luchko2021_Symmetry}
\[
    \int_0^t \phi_{\alpha,n}(t-\tau) \psi_{\alpha}(\tau) \du \tau = \frac{t^{n-1}}{(n-1)!}
    , \quad t > 0 ,
\]
that in the Laplace domain simply reads
$$\Phi_{\alpha,n}(s) s^{-sA(s)} = s^{-n} \, .$$
Then, the derivative kernel is simply obtained as
\[
    \phi_{\alpha,n}(t) \coloneqq {\mathcal L}^{-1}\Bigl( \Phi_{\alpha,n}(s) \, ; \, s\Bigr)
    , \quad 
    \Phi_{\alpha,n}(s)= s^{sA(s)-n} \, , 
\]
where one can clearly see that the necessary condition $\Phi_{\alpha,n}(s) \to 0$, as $s \to \infty$, is fulfilled.

\begin{remark}
Note that setting $n=1$ 
the entire discussion transposes into the analysis presented in the previous section for $0<\alpha(t)<1$.
\end{remark}

Therefore a more general variable-order derivative for $n-1<\alpha(t)<n$ is obtained as (see \cite[Definition 3.2]{Luchko2021_Symmetry}) 
\[
    \DS^{\alpha(t)}_0 f(t) \coloneqq \frac{\du^n}{\du t^n} \int_0^t \phi_{\alpha,n}(t-\tau) f(\tau) \du \tau - \sum_{j=0}^{n-1} f^{(j)}(0) \phi_{\alpha,j}(t) 
, \quad t \in [0,T] ,
\]
where, for $j=0,1,\dots,n-1$, one has that
\[
\phi_{\alpha,j}(t) = \frac{\du^n}{\du t^n} \int_0^t \phi_{\alpha,n}(t-\tau) \frac{\tau^j}{j!}\du \tau 
= \frac{\du^{n-j-1}}{\du t^{n-j-1}} \phi_{\alpha,n}(t).
\]
However, a more practical way of computing the functions $\phi_{\alpha,j}(t)$ 
relies on noting that $
    \phi_{\alpha,j}(t) \coloneqq {\mathcal L}^{-1} \bigl( s^{n-j-1}\Phi_{\alpha,n}(s) \, ; \, s \bigr)$.

It is also possible to provide a different characterization of $\DS^{\alpha(t)}_0 f(t)$ if $f(t)$ is sufficiently regular. Indeed, by iterating the procedure in Proposition \ref{prop:ScarpiDerivative}, one can conclude that if $f (t)$ is differentiable $n-1$ times in $[0,T]$ with $f^{(n-1)} \in AC[0,T]$, then
\[
    \DS^{\alpha(t)}_0 f(t) = \int_0^t \phi_{\alpha,n}(t-\tau) f^{(n)}(\tau) \du \tau , \quad t \in [0,T]  \, ,
\]
see \cite[Theorem 5]{Luchko2021_Mathematics} for details.
%We must warn readers that higher-order operators in this Section are not presented under a completely rigours mathematical theory since we  just aim to provide preliminary ideas. A more in-depth analysis is necessary, for instance, for more general order transition functions $\alpha(t):[0,T] \to (0,n)$.

\section{Concluding remarks}\label{S:FinalConsiderations}

This paper aims at making the first step toward reviving Scarpi's ideas on variable-order fractional calculus. We have framed these ideas in terms of the recently developed theory of generalized fractional calculus \cite{Luchko2021_Mathematics,Luchko2021_Symmetry,Kochubei2011,Luchko2020_FCAA} and we have shown one of the possible numerical approaches needed for handling these derivatives and related initial value problems. 

There are still many open problems that need to be addressed. For instance, despite the analysis presented here, an exact characterization of the proprieties that the transition function $\alpha(t)$ must satisfy  in order to generate a pair of suitable Scarpi's operators $\{\IS^{\alpha}_0,\DS^{\alpha}_0\}$ remains an issue requiring some attention. Further, the discussion presented in this work was limited to transition functions with values in either $(0,1)$ or $(n-1,n)$, however, considering transitions in $(0,n)$ could be of interest for some physical applications. Additionally, a precise investigation of the general structure of the eigenfunctions of the relaxation equation (\ref{eq:S_Relaxation}), and of their asymptotic properties, would prove invaluable for further physical applications. Lastly in our (incomplete) collection of open questions in Scarpi's theory, further efforts should be devoted to designing efficient numerical methods to solve more general fractional differential equations involving the these operators. 

To conclude, Scarpi's theory offers a brand new way of looking at variable-order processes in fractional calculus with a limitless potential  for applications in physics, engineering, and
other natural sciences.

\appendix
\section{A numerical method for the inversion of the LT}\label{S:NumericalInversionLT}

In this Appendix we provide a detailed  description of the method employed in the previous Sections to numerically invert the LT of the kernels $\phi_{\alpha}(t)$ and $\psi_{\alpha}(t)$  and of the solution of the relaxation equation (\ref{eq:S_Relaxation}).

The method is based on the main idea by Talbot \cite{Talbot1979} consisting in deforming,  in the formula for the inversion of the LT $F(s)$ of a function $f(t)$
\begin{equation}\label{eq:InverseLTBromwich}
f(t) = \frac{1}{2 \pi \iu}\int_{\sigma-\iu\infty}^{\sigma+\iu\infty} \eu^{s t} F(s)  \du s ,
\end{equation}
the  Bromwich line $(\sigma-\iu\infty,\sigma+\iu\infty)$ into a different contour ${\mathcal C}$ beginning and ending in the left complex half-plane. In this way it is possible to obtain an accurate approximation of the function $f(t) $ after applying a suitably chosen  quadrature rule along ${\mathcal C}$, since the strong oscillations of the exponential, and the resulting numerical instability, are avoided. 

This approach  was successively refined  by Weidemann and Trefethen \cite{WeidemanTrefethen2007} who  provided a detailed error analysis allowing to properly select the geometry of the contour ${\mathcal C}$ and the quadrature parameters in order to achieve any prescribed accuracy $\varepsilon>0$ (a tailored analysis for the ML function was successively proposed in \cite{GarrappaPopolizio2013} and applied in the context of ML with matrix arguments \cite{GarrappaPopolizio2018} as well). A further improvement was introduced in \cite{Garrappa2015_SIAM} with the aim o better handling LTs $F(s)$ with one or more singularities scattered in the complex plane. However. since in our examples we are faced with LTs $F(s)$ having just singularities at the origin or on the branch-cut, the original algorithm introduced in \cite{WeidemanTrefethen2007} turns out to be good enough. 

One of the most useful contours used to replace the Bromwich line in (\ref{eq:InverseLTBromwich}) is a parabolic-shaped contour described by the equation
\[
z(u) = \mu (\iu u+1)^2,\,\, -\infty < u < \infty,
\]
where $\mu>0$ is a parameter determining the abscissa where the parabola crosses the real axis and the concavity of the parabola. Although more efficient contours are available (with these regards we refer to the analysis in \cite{TrefethenWeidemanSchmelzer2006}), parabolas present the major advantage of  a very simple representation, depending on just one parameter, which simplifies the error analysis.

After deforming the Bromwich line into the parabolic contour $z(u)$, suitably chosen to encompass any possible singularity of $F(s)$, one obtains the equivalent formulation 
\begin{equation}\label{eq:InverseLTParabolic2}
f(t) = \frac{1}{2 \pi \iu}\int_{-\infty}^{+\infty} \eu^{z(u) t} F(z(u)) z'(u) \du u.
\end{equation}

Hence, the application of a trapezoidal rule with step-size $h$ on a sufficiently large truncated interval $[-hN,hN]$ leads to the approximation
\begin{equation}\label{eq:InverseLTParabolic}
f_{h,N}(t) = \frac{h}{2 \pi \iu}\sum_{k=-N}^{N} \eu^{z(u_k) t} F(z(u_k)) z'(u_k) 
, \quad u_k = hk .
\end{equation}

The choice of the three parameters $\mu$, $h$ and $N$ is essential to achieve an accurate approximation of $f(t)$ and it is driven by a detailed analysis of the error $|f(t) - f_{h,N}(t)|$. This in turn consists of  two main components: the discretization error (DE) and the truncation error (TE). By following the analysis in \cite{WeidemanTrefethen2007}, in absence of singularities of $F(s)$ (except for the branch-point singularity at the origin and the branch-cut placed, for convenience, on the negative real semi-axis) one can find that 
\[
\begin{aligned}
    |DE| &= {\mathcal O}\Bigl( \eu^{-2\pi/h} \Bigr) + 
    {\mathcal O}\Bigl( \eu^{- \pi^2/(\mu t h^2) + 2\pi/h} \Bigr) , \quad h \to 0, \\
    |TE| &= {\mathcal O}\Bigl( \eu^{\mu t \bigl(1-(hN)^2\bigr)}\Bigr) ,
    \quad h \to 0 \
\end{aligned}
\]

A more accurate analysis  takes into account the round-off error (RE) as well, for which (after exploiting $|z'(u)| = 2 \sqrt{\mu} \sqrt{|z(u))}|$) the following estimates hold \cite{Weideman2010}
\[
\begin{aligned}
    |\text{RE}| &\le \frac{\epsilon h}{\pi} \eu^{\mu t} \sum_{k=0}^{N} |F(z(u_k))| |z'(u_k)|
    = \frac{2 \epsilon \sqrt{\mu} h}{\pi} \eu^{\mu t} \sum_{k=0}^{N} |\hat{F}(z(u_k))| \\
    &\approx \epsilon \eu^{\mu t}
    \frac{2 \sqrt{\mu}}{\pi} \int_0^{Nh} |\hat{F}(s)| \du s , \\
\end{aligned}
\]
where $\epsilon$ is the precision machine and $\hat{F}(s) = F(s) s^{\frac{1}{2}}$. Obviously, the analysis needs to be customized according to the specific LT $F(s)$ which must be inverted. If $\hat{F}(s)$ is assumed to have a moderate growth and $Nh$ is in general not large (in practice very often it is $h={\mathcal O}\bigl(N^{-1}\bigr)$ one can neglect the integral in the estimate of $RE$ (as well as the $2\sqrt{\mu}/\pi$ term) and just assume
$
    |RE| \approx \epsilon \eu^{\mu t} 
$.

Optimal parameters $\mu$, $h$ and $N$ can be now  obtained after balancing the three different errors and imposing that they are proportional to a given prescribed accuracy which, to simplify the analysis and at the same time ensure accurate results, we select at the same level of the precision machine $\epsilon\approx 2.22\times 10^{-16}$. Therefore,  after imposing that $|DE| \approx |TE| \approx |RE| \approx \varepsilon$ asymptotically as $h \to 0$, and denoting $L=-\log\epsilon$, the balancing of the three errors leads to
\[
N = \frac{4L}{3\pi} 
, \quad
\mu = \frac{L^3}{4 t \pi^2 N^2}
, \quad
h = \frac{2\pi}{L} + \frac{L}{2 \pi N^2} .
\]

\begin{remark}
In the above analysis we have assumed a moderate growth of $\hat{F}(s)$ as $s\to \infty$. With respect to the transition $\alpha(t)$ considered in our examples this assumption is truly reasonable in order to compute $\psi_{\alpha}(t)$ or the solution $y(t)$ of the relaxation equation (\ref{eq:S_Relaxation}) but could be too much optimistic for the evaluation of $\phi_{\alpha}(t)$ which is expected to have a more sustained growth. Although we have obtained reasonable results the same, we think that a more detailed analysis is necessary if one aims to compute $\phi_{\alpha}(t)$ with high accuracy.  
\end{remark}

In the following we report the few lines of a Matlab code for the numerical inversion of  the LT $F(s)$ on a vector of points $t$. The code is optimized to evaluate just functions $f(t)$ with real values. The LT $F(s)$ is assumed not to have singularities except a possible one at the origin.

\begin{lstlisting}[style=Matlab-editor,float,floatplacement=H]
L = -log(eps) ; 
N = ceil(4*L/3/pi) ; 
h = 2*pi/L + L/2/pi/N^2 ; 
p = L^3/4/pi^2/N^2 ;
u = (0:N)*h ;
f = zeros(size(t)) ;
for n = 1 : length(t)
    mu = p/t(n) ;
    z = mu*(u*1i+1).^2 ; z1 = 2*mu*(1i-u) ;
    G = exp(z*t(n)).*F(z).*z1 ;
    f(n) = (imag(G(1))/2+sum(imag(G(2:N+1))))*h/pi ;
end
\end{lstlisting}

\section*{Acknowledgments}

The work of R.Garrappa is  supported by INdAM under a GNCS-Project 2020. The work of A.Giusti is supported by the Natural Sciences and Engineering Research Council of
Canada (Grant No. 2016-03803 to V. Faraoni) and by Bishop’s University. The work of A.Giusti and F.Mainardi has been carried out in the framework of the activities of the Italian National Group for Mathematical Physics [Gruppo Nazionale per la Fisica Matematica (GNFM), Istituto Nazionale di Alta Matematica (INdAM)].

%\bibliographystyle{siamplain}
%\bibliography{Scarpi_Biblio}

%\end{document}

\end{document}